\documentclass[a4paper]{amsart}
\usepackage{latexsym}\usepackage{ifthen}\usepackage[leqno]{amsmath}
\usepackage{enumerate}\usepackage{calc}\usepackage{hyphenat}
\usepackage{amstext,amsbsy,amsopn,amsthm,amsgen}
\usepackage{amsfonts,amscd,amsxtra,upref}
\usepackage{mathrsfs}\usepackage{euscript}\usepackage{amssymb}
\usepackage[OT4]{fontenc}
\swapnumbers
\theoremstyle{plain}
\makeatletter\@namedef{subjclassname@2010}{\textup{2010} Mathematics Subject Classification}
\makeatother
\newtheorem{Thm}{Theorem}[section]
\newtheorem{Lem}[Thm]{Lemma}
\newtheorem{Cor}[Thm]{Corollary}
\newtheorem{Pro}[Thm]{Proposition}
\newtheorem{Cnj}[Thm]{Conjecture}
\theoremstyle{definition}
\newtheorem{Def}[Thm]{Definition}
\newtheorem{Exm}[Thm]{Example}

\theoremstyle{remark}
\newtheorem{Rem}[Thm]{Remark}
\numberwithin{equation}{section}

\newcommand{\ITE}[3]{\ifthenelse{#1}{#2}{#3}}\newcommand{\ITEE}[4][]{\ITE{\equal{#2}{#3}}{#4}{#1}}
 \ITE{\isundefined{\texorpdfstring}}{\newcommand{\texorpdfstring}[2]{#1}}{}

\newcommand{\myData}[1][]{
 \author[P.\ Niemiec]{Piotr Niemiec}
 \address{\ITEE{#1}{*}{P.\ Niemiec{}\\}Instytut Matematyki\\
  Wydzia\l{} Matematyki i~Informatyki\\Uniwersytet Jagiello\'{n}ski\\
  ul.\ \L{}ojasiewicza 6\\30-348 Krak\'{o}w\\Poland}
 \email{piotr.niemiec@uj.edu.pl}
 }

\newenvironment{cor}[2][]{\ITEE[{\begin{Cor}[#1]}]{#1}{}{\begin{Cor}}\label{cor:#2}}{\end{Cor}}
\newenvironment{dfn}[2][]{\ITEE[{\begin{Def}[#1]}]{#1}{}{\begin{Def}}\label{def:#2}}{\end{Def}}
\newenvironment{exm}[2][]{\ITEE[{\begin{Exm}[#1]}]{#1}{}{\begin{Exm}}\label{exm:#2}}{\end{Exm}}
\newenvironment{lem}[2][]{\ITEE[{\begin{Lem}[#1]}]{#1}{}{\begin{Lem}}\label{lem:#2}}{\end{Lem}}
\newenvironment{pro}[2][]{\ITEE[{\begin{Pro}[#1]}]{#1}{}{\begin{Pro}}\label{pro:#2}}{\end{Pro}}
\newenvironment{rem}[2][]{\ITEE[{\begin{Rem}[#1]}]{#1}{}{\begin{Rem}}\label{rem:#2}}{\end{Rem}}
\newenvironment{thm}[2][]{\ITEE[{\begin{Thm}[#1]}]{#1}{}{\begin{Thm}}\label{thm:#2}}{\end{Thm}}
\newenvironment{cnj}[2][]{\ITEE[{\begin{Cnj}[#1]}]{#1}{}{\begin{Cnj}}\label{cnj:#2}}{\end{Cnj}}

\newcommand{\COR}[2][!]{\ITEE{#1}{!}{Corollary~}\ITEE{#1}{s}{Corollaries~}\textup{\ref{cor:#2}}}
\newcommand{\DEF}[2][!]{\ITEE{#1}{!}{Definition~}\ITEE{#1}{s}{Definitions~}\textup{\ref{def:#2}}}
\newcommand{\EXM}[2][!]{\ITEE{#1}{!}{Example~}\ITEE{#1}{s}{Examples~}\textup{\ref{exm:#2}}}
\newcommand{\LEM}[2][!]{\ITEE{#1}{!}{Lemma~}\ITEE{#1}{s}{Lemmas~}\textup{\ref{lem:#2}}}
\newcommand{\PRO}[2][!]{\ITEE{#1}{!}{Proposition~}\ITEE{#1}{s}{Propositions~}\textup{\ref{pro:#2}}}
\newcommand{\THM}[2][!]{\ITEE{#1}{!}{Theorem~}\ITEE{#1}{s}{Theorems~}\textup{\ref{thm:#2}}}

\newcommand{\CCC}{\mathbb{C}}
\newcommand{\DDD}{\mathbb{D}}
\newcommand{\NNN}{\mathbb{N}}

\newcommand{\ZZZ}{\mathbb{Z}}
\newcommand{\DdD}{\EuScript{D}}
\newcommand{\NnN}{\EuScript{N}}
\newcommand{\RrR}{\EuScript{R}}
\newcommand{\TtT}{\EuScript{T}}
\newcommand{\Bb}{\mathfrak{B}}
\newcommand{\Uu}{\mathfrak{U}}
\newcommand{\mM}{\mathfrak{m}}
\newcommand{\ccC}{\mathscr{C}}
\newcommand{\ggG}{\mathscr{G}}
\newcommand{\hhH}{\mathscr{H}}
\newcommand{\rrR}{\mathscr{R}}
\newcommand{\vvV}{\mathscr{V}}
\newcommand{\wwW}{\mathscr{W}}
\newcommand{\zzZ}{\mathscr{Z}}

\newcommand{\ueR}{\textup{\textsf{R}}}

\newcommand{\dd}{\colon}
\newcommand{\df}{\stackrel{\textup{def}}{=}}
\newcommand{\dint}[1]{\,\textup{d} #1}
\newcommand{\epsi}{\varepsilon}
\newcommand{\geqsl}{\geqslant}
\newcommand{\leqsl}{\leqslant}
\newcommand{\op}{\textup{\textsf{op}}}
\newcommand{\scalar}[2]{\left\langle#1,#2\right\rangle}
\newcommand{\scalarr}{\langle\cdot,\mathrm{-}\rangle}
\newcommand{\varempty}{\varnothing}

\newcommand{\OPN}[1]{\operatorname{#1}}
\newcommand{\card}{\operatorname{card}}
\newcommand{\lin}{\operatorname{lin}}
\newcommand{\RE}{\operatorname{Re}}

\newcommand{\TFCAE}{The following conditions are equivalent:}
\newcommand{\tfcae}{the following conditions are equivalent:}

\begin{document}

\title[Positive Hankel operators and kernels]%
 {Positive Hankel operators,\\positive definite kernels and related topics}
\myData
\dedicatory{Dedicated to Professor Franciszek Hugon Szafraniec---the main of my
 two scientific fathers,\\on the occasion of his 80th birthday}
\begin{abstract}
It is shown that a positive (bounded linear) operator on a Hilbert space with
trivial kernel is unitarily equivalent to a Hankel operator that satisfies
double positivity condition if and only if it is non-invertible and has simple
spectrum (that is, if this operator admits a cyclic vector). More generally, for
an arbitrary positive (bounded linear) operator \(A\) on a Hilbert space \(H\)
with trivial kernel the collection \(\vvV(A)\) of all linear isometries \(V\dd H
\to H\) such that \(A V\) is positive as well is investigated. In particular,
operators \(A\) such that \(\vvV(A)\) contains a pure isometry with a given
deficiency index are characterized. Some applications to unbounded positive
self-adjoint operators as well as to positive definite kernels are presented.
In particular, positive definite matrix-type square roots of such kernels are
studied and kernels that have a unique such root are characterized. The class of
all positive definite kernels that have at least one such a square root is also
investigated.
\end{abstract}
\maketitle

\section{Introduction}

In \cite{mpt} the authors characterized (in the language of the multiplicity
theory of separable Hilbert space self-adjoint operators) bounded self-adjoint
operators that are unitarily equivalent to Hankel. This is a deep result whose
proof is difficult and long. For a bounded \textbf{positive} operator \(A\) on
a separable Hilbert space the following two theorems immediately follow:
\begin{itemize}
\item if \(A\) is unitarily equivalent to a Hankel operator, the essential
 supremum of the multiplicity function of \(A\) does not exceed \(2\);
\item if the essential supremum of the multiplicity function of \(A\) does not
 exceed \(1\), \(A\) is unitarily equivalent to a Hankel operator.
\end{itemize}
Hankel operators can be defined in a few equivalent ways. One of them,
appropriate to our investigations, reads as follows: a bounded operator \(A\dd
\ell_2 \to \ell_2\) is \textit{Hankel} if \(A S = S^* A\) where \(S\) is
the standard unilateral shift (that is, \(S\) is a linear isometry satisfying
\(S e_n = e_{n+1}\) for any \(n \geqsl 0\) where \(e_0,e_1,\ldots\) is
the canonical orthonormal basis of \(\ell_2\)). So, \(A\) is both Hankel and
self-adjoint iff both \(A\) and \(AS\) are self-adjoint. When in the last
condition we replace the adjective \textsl{self-adjoint} by \textsl{positive},
we obtain the so-called \textit{double positivity condition} \cite{g-p}: \(A\)
is a Hankel operator satisfying double positivity condition if both \(A\) and
\(AS\) are positive. Having this notion in mind, a natural question (related to
the main topic of the aforementioned paper) arises of when a (bounded) positive
operator on a separable Hilbert space is unitarily equivalent to a Hankel
operator satisfying double positivity condition. In the present paper we answer
this question under the additional assumption that the operator in question has
trivial kernel. Our main result reads as follows (below \(\RrR(T)\) denotes
the range of an operator \(T\)).

\begin{thm}{main}
For a positive bounded operator \(A\dd H \to H\) with trivial kernel and
a cardinal \(\mM > 0\) \tfcae
\begin{enumerate}[\upshape(i)]
\item there exists a pure isometry \(V\dd H \to H\) such that \(AV\) is positive
 and \[\dim \RrR(V)^{\perp} = \mM;\]
\item \(A\) is non-invertible, \(\dim(H) = \max(\mM,\aleph_0)\) and
 an appropriate condition of the following three is fulfilled:
 \begin{itemize}
 \item[\((\alpha)\)] \(\mM < \aleph_0\) and the essential supremum of
  the multiplicity function of \(A\) does not exceed \(\mM\); or
 \item[\((\beta)\)] \(\mM = \aleph_0\); or
 \item[\((\gamma)\)] \(\mM > \aleph_0\) and there exists a closed linear
  subspace \(Z\) of \(H\) such that \(Z \cap \RrR(A) = \{0\}\) and \(\dim(Z) =
  \mM\).
 \end{itemize}
\end{enumerate}
In particular, a bounded one-to-one linear operator is unitarily equivalent to
a Hankel operator that satisfies double positivity condition iff it is positive,
non-invertible and has simple spectrum.
\end{thm}

\THM{main} shows (in particular) that for infinite \(\mM\) the property (i)
above depends only on the range of the operator \(A\) (i.e., if two positive
operators \(A\) and \(B\) have trivial kernels and their ranges coincide, then
either both \(A\) and \(B\) satisfy (i) or none of them). In \THM{pure} below we
gather conditions on a dense operator range \(\RrR\) in a Hilbert space \(H\)
related to the foregoing statement \((\gamma)\)---that is, conditions equivalent
to the existence of a closed linear subspace \(Z\) of \(H\) such that \(Z \cap
\RrR = \{0\}\) and \(\dim(Z) = \dim(H)\).\par
Our proofs are independent of the results from \cite{mpt} and are much simpler.
We use basics of the operator theory and of the spectral theory of self-adjoint
operators.\par
Another topic we deal with in this paper is related to Hilbert space reproducing
(that is, positive definite) kernels. They are a useful tool in both Hilbert
space theory and complex analysis (where they are known as Bergman kernels).
Since the seminal paper of Aronszajn \cite{aro}, positive definite kernels are
a subject of an intristic theory. Although Bergman \cite{bm1,bm2} is considered
by a sizeable mathematical community as the father of that theory, it is
Zaremba's work \cite{zar}, published 15 years earlier than the first Bergman's
on kernels, where the reproducing property (without any name) appeared for
the first time---see, e.g., \cite{aro} or \cite{sz3}. We wish to emphasize
Zaremba's contribution to the theory by calling him its \textsl{forefather}.\par
As positive definite kernels naturally generalize positive matrices (for complex
square matrices can be seen as kernels defined on finite sets), it is natural to
investigate various properties of such matrices and recognize those of them that
inhere in all such kernels. In the present paper (in Section~5) we characterize
those kernels which have the so-called matrix-type square root. To be more
precise, we introduce the following

\begin{dfn}{root}
Let \(K\dd X \times X \to \CCC\) and \(L\dd X \times X \to \CCC\) be two
positive definite kernels. \(K\) is said to be a \textit{positive definite
matrix-type square root} (for short: a \textit{pdms} root) of \(L\) if for all
\(x,z \in X\):
\begin{equation}\label{eqn:matrix}
L(x,z) = \sum_{y \in X} K(x,y) K(y,z).
\end{equation}
\end{dfn}

(More on the above notion can be found in Section~5.) A classical result from
matrix theory (or, more generally, from bounded Hilbert space operator theory)
says that any positive matrix has a \underline{unique} positive square root.
A natural question arises as to how far this result extends in the realm of
reproducing kernels. Our main result in this direction reads as follows. (Below
we write ``\(K \ll L\)'' to express that \(L-K\) is a positive definite kernel;
and \(\delta_X\) is a kernel on \(X\) such that \(\delta_X(x,y) = 1\) if \(y =
x\) and \(\delta_X(x,y) = 0\) otherwise.)

\begin{thm}{uniq}
For a positive definite kernel \(K\dd X \times X \to \CCC\) \tfcae
\begin{enumerate}[\upshape(i)]
\item \(K\) has a unique positive definite matrix-type square root;
\item each positive definite kernel \(L\dd X \times X \to \CCC\) such that
 \(L \ll K\) has a positive definite matrix-type square root;
\item \(K \ll c \delta_X\) for some constant \(c > 0\).
\end{enumerate}
\end{thm}

(Note that the equivalence of conditions (i)--(iii) above implies that each
kernel \(L\) appearing in (ii) has in fact a unique pdms root.) We underline
here that condition (i) above says about both the existence and the uniqueness
of pdms roots. In Section~5 we also give equivalent conditions for a positive
definite kernel to have at least one such root.\par
For more information on reproducing kernels consult \cite{bm2} or \cite{sai}.
Modern expositions can be found in \cite{bta} or \cite{sz1,sz2}.\par
The paper is organized as follows. In Section~2 we study in greater detail
the collections \(\vvV(A)\) and \(\zzZ(A)\) introduced above, and give a proof
of \THM{main}. Next section contains further conditions (not listed in
\THM{main}) equivalent to condition (i) of that theorem. The reader can also
find there a full description of all possible unitary equivalence types of
isometries from \(\vvV(A)\) (see \THM{card} below). In the fourth part
the results of Section~2 are applied to unbounded operators. We prove that any
truly unbounded positive self-adjoint operator is the absolute value of some
positive closed operator that is not self-adjoint (this serves as a criterion
for a boundedness of positive self-adjoint operators)---consult \THM{unbd}. We
also show that \textbf{all} closable operators are (in a certain sense)
``conditionally'' weakly continuous (see \LEM{closed} therein). Last, fifth,
part is devoted to the notion of a pdms root (for reproducing kernels)
introduced above. We gather equivalent conditions for a positive definite kernel
to have at least one pdms root (\THM{exists}), prove \THM{uniq} and study in
greater detail kernels having such roots (see, e.g., \THM{pdms}). In particular,
among all such roots (of a fixed kernel) we distinguish one of them which can be
seen as (unique) ``self-adjoint'' (\COR{sa}). We also show that if a kernel is
a pdms root of some other kernel, then it automatically has a pdms root (item
(V) of \THM{pdms}). This property enables one to define (positive definite
matrix-type) roots of higher degrees. Apart from the results of Section~4,
proofs presented in the last part invoke the machinery of unbounded symmetric
operators with Friedrichs' theorem \cite{fri} on extending positive operators as
the main of them.

\subsection*{Notation and terminology}

Throughout this paper all Hilbert spaces are non-trivial and complex, and \(H\)
denotes one of them. By \(\dim(H)\) we denote the Hilbert space dimension of
\(H\), that is, \(\dim(H)\) is the cardinality of an orthonormal basis of \(H\).
The scalar product of \(H\) will be denoted by \(\scalarr_H\). All operators are
linear, act between Hilbert spaces and have dense domains. A linear subspace
generated by a set \(F\) is denoted by \(\lin(F)\) and \(\overline{\lin}(F)\)
stands for the closure of \(\lin(F)\). For any non-empty set \(X\) we use
\(\ell_2(X)\) to denote the Hilbert space of all square-summable complex-valued
functions on \(X\) equipped with the standard inner product. More precisely,
\(f\dd X \to \CCC\) belongs to \(\ell_2(X)\) if \(\sum_{x \in X} |f(x)|^2 <
\infty\); and for \(u, v \in \ell_2(X)\), \(\scalar{u}{v}_{\ell_2(X)} =
\sum_{x \in X} u(x) \overline{v(x)}\). The \textit{canonical} basis of
\(\ell_2(X)\) consists of functions \(e_x\) (where \(x\) runs over all elements
of \(X\)) of the form: \(e_x(x) = 1\) and \(e_x(y) = 0\) for \(y \neq x\). For
simplicity, we will denote by \(\ell_{fin}(X)\) the linear span of the canonical
basis (so, \(f\dd X \to \CCC\) belongs to \(\ell_{fin}(X)\) iff the set \(\{x
\in X\dd\ f(x) \neq 0\}\) is finite).\par
We use \(\Bb(H)\) and \(\Uu(H)\) to denote, respectively, the \(C^*\)-algebra of
all bounded operators on \(H\) and the group of all unitary operators on \(H\);
\(I = I_H\) is used to denote the unit of \(\Uu(H)\), and \(\Bb_+(H)\) stands
for the collection of all bounded positive operators \textbf{with trivial
kernel}. (In particular, each member of \(\Bb_+(H)\) is a self-adjoint operator
with dense range.) For two self-adjoint operators \(A, B \in \Bb(H)\) we write
\(A \leqsl B\) or \(B \geqsl A\) if the operator \(B-A\) is positive; that is,
if \(\scalar{(B-A)x}{x}_H \geqsl 0\) for any \(x \in H\). By
a \textit{contraction} we mean a bounded operator between Hilbert spaces whose
operator norm is not greater than \(1\).\par
Whenever \(T\) is an operator, we use \(\DdD(T)\), \(\NnN(T)\), \(\RrR(T)\) and
\(\Gamma(T)\) to denote, respectively, the domain, the kernel, the range and
the graph of \(T\). In addition, \(\overline{\RrR}(T)\) denotes the closure of
\(\RrR(T)\). The operator \(T\) is \textit{closed} if \(\Gamma(T)\) is closed
in the product of Hilbert spaces between which \(T\) acts. \(T\) is
\textit{closable} if the closure of \(\Gamma(T)\) is the graph of
an operator---in that case \(\bar{T}\) denotes the unique operator whose graph
coincides with the closure of \(\Gamma(T)\); \(\bar{T}\) is called
the \textit{closure} of \(T\) and \(\DdD(T)\) a \textit{core} of \(\bar{T}\).
For any closed linear subspace \(K\) of \(H\), \(P_K\) stands for the orthogonal
projection from \(H\) onto \(K\).\par
Basic facts on the multiplicity theory for bounded self-adjoint operators on
separable Hilbert spaces can be found in \S10 of Chapter~IX in \cite{co1}. To
undestand the present paper it is sufficient to know the following result, which
will be used several times in this paper: \textit{the essential supremum of
the multiplicity function of a bounded self-adjoint operator \(A\) acting on
a separable Hilbert space does not exceed \(n \in \{1,2,\ldots\}\) iff \(A\) can
be decomposed as the direct sum of at most \(n\) self-adjoint operators each of
which has a cyclic vector}. Recall that a vector \(v \in H\) is
a \textit{cyclic} vector for a self-adjoint operator \(B \in \Bb(H)\) if \(H =
\overline{\lin}(\{B^n v\dd\ n \geqsl 0\})\).\par
Any closed operator \(T\dd \DdD(T) \to K\) (where \(\DdD(T)\) is a dense
subspace of \(H\)) admits the so-called \textit{polar decomposition} which has
the form
\begin{equation}\label{eqn:polar}
T = Q A
\end{equation}
where \(A\dd \DdD(T) \to H\) is positive self-adjoint in \(H\) and \(Q\dd H \to
K\) is a partial isometry. Moreover, the above \(Q\) and \(A\) are uniquely
determined by \eqref{eqn:polar} and condition \(\NnN(Q) = \NnN(T)\). The above
operator \(A\) satisfies \(A^2 = T^*T\) and is called the \textit{absolute
value} of \(T\) and denoted by \(|T|\). For the details see Theorem~7.20 in
\cite{wei}.\par
For any \(A \in \Bb_+(H)\) we denote by \(\vvV(A)\) and \(\zzZ(A)\)
the collections, respectively, of all isometries \(V \in \Bb(H)\) such that \(AV
\in \Bb_+(H)\), and of all closed linear subspaces \(Z\) of \(H\) such that \(Z
\cap \RrR(A) = \{0\}\). Additionally---for simplicity---for any set \(F\) in
\(H\), \([F]_A\) stands for the set \(\overline{\lin}(\bigcup_{n=0}^{\infty}
A^n(F))\); that is, \([F]_A\) is the smallest closed linear subspace of \(H\)
that contains \(F\) and is invariant under \(A\).\par
All necessary notions concerning reproducing kernels are introduced and
discussed in Section~5.

\section{Isometries of class \(\vvV(A)\)}

In this section \(A\) is fixed and denotes a member of \(\Bb_+(H)\). We begin
with

\begin{pro}{iso-lin}
A function \[\Upsilon_A\dd \vvV(A) \ni V \mapsto \RrR(V)^{\perp} \in \zzZ(A)\]
is a well defined bijection.
\end{pro}
\begin{proof}
First assume that \(V \in \vvV(A)\) and put \(Z = \RrR(V)^{\perp}\). Then \(V^*A
= AV\), hence \(\{0\} = \NnN(AV) = \NnN(V^*A) = A^{-1}(\NnN(V^*)) = A^{-1}(Z
\cap \RrR(A))\), which shows that \(Z \cap \RrR(A) = \{0\}\); that is,
\(\Upsilon_A\) is well defined. Now assume that also \(W \in \vvV(A)\) satisfies
\(\RrR(W)^{\perp} = Z\). Define an operator \(U \in \Uu(H)\) by \(U \df
V^{-1}W\) and observe that \(W = VU\). It follows from the assumptions that \(B
\df AV\) and \(C \df BU\) are bounded positive operators such that \(C^2 = CC^*
= (BU)(U^*B) = B^2\). Since bounded positive operators have unique positive
square roots, we infer that \(B = C\). Since \(B\) has trivial kernel, we get
\(U = I\) and thus \(W = V\). In other words, \(\Upsilon_A\) is one-to-one.\par
Now take any \(Z \in \zzZ(A)\) and define \(D \in \Bb_+(H)\) as the (unique)
positive square root of \(A(I - P_Z)A\). Note that then \(D^2 = ((I-P_Z)A)^*
((I-P_Z)A)\). It follows that \(\NnN(D) = \NnN((I-P_Z)A) = A^{-1}(Z) = \{0\}\)
and hence the range of \(D\) is dense in \(H\). We also infer that \(D =
|(I-P_Z)A|\) and hence---by the properties of the polar decomposition:
\begin{equation}\label{eqn:range}
\RrR(D) = \RrR(((I-P_Z)A)^*) = A(Z^{\perp}).
\end{equation}
(The above formula will be used in the proof of the next result.) Further, since
\begin{equation}\label{eqn:sum}
A^2 = D^2 + A P_Z A,
\end{equation}
we see that \(D^2 \leqsl A^2\) and thus, by \cite{dou} (see also Theorem~2.1
in \cite{f-w}), there is a contraction \(V \in \Bb(H)\) such that \(D = AV\).
Observe that then \(D = V^* A\) and \(\RrR(V^*)\) is dense in \(H\). Moreover,
it follows from \eqref{eqn:sum} that \(AVV^*A = A(I-P_Z)A\). Since \(A\) has
dense range and trivial kernel, we get that \(VV^* = I-P_Z\). Consequently,
\(V\) is a partial isometry with \(\RrR(V) = \RrR(I-P_Z) = Z^{\perp}\). But
the range of \(V^*\) is dense in \(H\) and hence \(V\) is an isometry such that
\(\RrR(V)^{\perp} = Z\). A note that \(V \in \vvV(A)\) (because \(AV = D\))
completes the proof.
\end{proof}

For any \(Z \in \zzZ(A)\) we use \(W_Z\) to denote a unique \(V \in \vvV(A)\)
such that \(\RrR(V)^{\perp} = Z\) (cf.\ \PRO{iso-lin}). Of course, \(W_{\{0\}} =
I\) is the only unitary operator in \(\vvV(A)\).\par
Recall that any isometry \(V \in \Bb(H)\) induces a unique decomposition \(H =
H_u \oplus H_p\) (called \textit{Wold's decomposition}) of the space \(H\) such
that both \(H_u\) and \(H_p\) are invariant under \(V\), \(V\restriction{H_u}\)
is a unitary operator on \(H_u\) and \(\{0\}\) is the only closed linear
subspace \(K\) of \(H_p\) such that \(V(K) = K\). Each of the spaces \(H_u\) and
\(H_p\) can be trivial. The restrictions of \(V\) to \(H_u\) and \(H_p\) are
called by us, respectively, the \textit{unitary} and \textit{pure} parts of
\(V\). We call the isometry \(V\) \textit{pure} if \(H = H_p\). It is well-known
(and easy to prove) that \(H_u = \bigcap_{n=0}^{\infty} \RrR(V^n)\) and \(H_p =
\bigoplus_{n=0}^{\infty} V^n(\RrR(V)^{\perp})\). Any pure isometry \(W\) is
unitarily equivalent to the direct sum of \(\alpha\) copies of the (standard)
unilateral shift where \(\alpha = \dim \RrR(W)^{\perp}\). The cardinal
\(\alpha\) defined above is called by us the \textit{deficiency index} of
the isometry \(W\). For the proofs of the above facts consult, e.g., Chapter~1
in \cite{r-r} (therein pure isometries are called shifts and the deficiency
index of a pure isometry is called its multiplicity).\par
Now we describe Wold's decompositions of members of \(\vvV(A)\).

\begin{thm}{Wold}
Let \(Z \in \zzZ(A)\), \(V = W_Z\) and \(H = H_u \oplus H_p\) be the Wold's
decomposition induced by \(V\). Then:
\begin{itemize}
\item \(H_p = [Z]_A\);
\item \(H_u = \NnN(V-I)\).
\end{itemize}
\end{thm}
\begin{proof}
We continue the notation introduced in the proof of \PRO{iso-lin}: let \(D\)
denote the positive square root of \(A(I-P_Z)A\). Then the formulas
\eqref{eqn:range} and \eqref{eqn:sum} are valid. Moreover, we have \(D = AV\)
and \(\NnN(D) = 0\). Define a bounded operator \(T\dd H \to H \oplus Z\) by
\begin{equation}\label{eqn:aux1}
Tx = Dx \oplus P_Z Ax \qquad (x \in H).
\end{equation}
We claim that \(\RrR(T)\) is dense in \(H \oplus Z\). To convince ourselves of
that, fix \(u \in H\) and \(v \in Z\) such that \(u \oplus v \perp \RrR(T)\).
This means that for any \(x \in H\),
\[0 = \scalar{Tx}{u \oplus v}_{H \oplus Z} = \scalar{x}{Du}_H + \scalar{x}{Av}_H
= \scalar{x}{Du+Av}_H.\]
So, \(Du = -Av\), but \(Du \in A(Z^{\perp})\) (by \eqref{eqn:range}) and \(Av
\in A(Z)\). Since \(A\) is one-to-one, we get \(Du = 0 = Av\) and therefore
\(u = v = 0\) (as both \(A\) and \(D\) have trivial kernels).\par
Further, \eqref{eqn:aux1} combined with \eqref{eqn:sum} yields that \(\|Tx\|^2 =
\|Ax\|^2\) for any \(x \in H\). Since both \(T\) and \(A\) have dense ranges, we
infer that there exists a (unique) unitary operator \(Q\dd H \oplus Z \to H\)
such that
\begin{equation}\label{eqn:Q}
QT = A.
\end{equation}
Consequently, \(T = Q^{-1}A\) and hence (by \eqref{eqn:aux1}) \(D = PQ^{-1}A\)
where \(P\dd H \oplus Z \to H\) is the projection onto the first coordinate. But
\(D = V^*A\) and \(A\) has dense range. So, \(V^* = PQ^{-1}\). Equivalently,
\(V = QP^*\). In other words,
\begin{equation}\label{eqn:aux3}
Vx = Q(x \oplus 0) \qquad (x \in H).
\end{equation}
For simplicity, denote by \(E\) the subspace \([Z]_A\). Since \(A^2(E) \subset
E\) and \((AP_ZA)(E) \subset A(Z) \subset E\), we infer from \eqref{eqn:sum}
that \(D^2(E) \subset E\) and that \(D(E) \subset E\) as well.\par
Now fix \(x \in E^{\perp}\). Since \(x \perp A^n(Z)\), it follows from
the self-adjointness of \(A\) that \(A^n x \perp Z\) for any \(n \geqsl 0\).
Hence \(P_Z A^n x = 0\) and by a simple induction argument applied to
\eqref{eqn:sum} we get \(A^{2n}x = D^{2n}x\) for any \(n \geqsl 0\). Since there
is a sequence of polynomials \(p_1,p_2,\ldots\) such that \(p_n(A^2) \to A\)
and \(p_n(D^2) \to D\) in the operator norm as \(n \to \infty\), we obtain
\(A x = D x\). So, thanks to \eqref{eqn:aux3}, \eqref{eqn:aux1} and
\eqref{eqn:Q}, \(V(Ax) = V(Dx) = Q(Dx \oplus 0) = Q(Tx) = Ax\). But
\(A(E^{\perp}) \subset E^{\perp}\) and \(A\) is one-to-one, thus
\(A(E^{\perp})\) is dense in \(E^{\perp}\). We infer that \(V z = z\) for all
\(z \in E^{\perp}\). Since \(V\) is an isometry, \(V(E) \subset E\). So, to
finish the whole proof, it is sufficient to show that \(V\restriction{E}\) is
a pure isometry on \(E\) (recall that \(1\) is an eigenvalue of no pure
isometry). To this end, we restrict our further considerations to the space
\(E\) (note that \(E\) is invariant for all \(A\), \(D\) and \(V\) and that
\(Z \subset E\) and \(A\restriction{E} \cdot V\restriction{E} =
D\restriction{E}\)). In other words, we assume that \(H = E\). Although
everywhere below we will identify \(A\), \(V\) and \(D\) with their restrictions
to \(E\), we shall write \(E\) instead of \(H\) to avoid confusion.\par
Let
\begin{equation}\label{eqn:aux4}
E = E_u \oplus E_p
\end{equation}
be the Wold's decomposition induced by \(V\). We only need to show that \(E_p =
E\). To this end, put \(U \df V\restriction{E_u} \in \Uu(E_u)\) and \(S \df
V\restriction{E_p} \in \Bb(E_p)\) and note that \(S\) is a pure isometry on
\(E_p\) and \(Z \subset E_p\) (as \(Z = \RrR(V)^{\perp}\)). Represent \(A\) as
a block matrix \(A = \begin{pmatrix}B & X\\X^* & C \end{pmatrix}\) with respect
to the decomposition \eqref{eqn:aux4} (that is, \(B\dd E_u \to E_u\), \(X\dd E_p
\to E_u\) and \(C\dd E_p \to E_p\)). Then \(B \in \Bb_+(E_u)\), \(C \in
\Bb_+(E_p)\) and \(D = A V = \begin{pmatrix}BU & XS\\X^*U & CS\end{pmatrix}\).
Since \(D \in \Bb_+(E)\), we conclude that \(BU \in \Bb_+(E_u)\) and
\begin{equation}\label{eqn:aux5}
UX\ (= (X^*U)^*) = XS.
\end{equation}
Then \(B^2 = (BU) (BU)^* = (BU)^2\). Consequently (by the uniqueness of
the positive square root), \(B = BU\) and hence \(U = I\). So, \eqref{eqn:aux5}
transforms to \(X(I-S) = 0\). Since \(S\) is a pure isometry, the range of
\(I-S\) is dense in \(E_p\). We conclude that \(X = 0\). So, \(A = B \oplus C\)
and \(E_p\) is invariant under \(A\). Hence \(E = [Z]_A \subset E_p\) and we
are done.
\end{proof}

The above result shows that for \(Z \in \zzZ(A)\) the structure of the isometry
\(W_Z\) is completely determined by two cardinal numbers: \(\alpha(Z) \df
\dim(Z)\) and \(\beta_A(Z) = \dim([Z]_A^{\perp})\). It is a natural question of
when it may happen (for a fixed operator \(A\)) that \(\beta_A(Z) = 0\) for some
\(Z\); that is, when \(\vvV(A)\) contains a pure isometry with a pre-set
deficiency index. This question is fully answered in the next three
propositions.

\begin{pro}{fin}
For any \(n \in \{1,2,\ldots\}\) \tfcae
\begin{enumerate}[\upshape(i)]
\item there exists a pure isometry \(V \in \vvV(A)\) with deficiency index
 \(n\);
\item \(H\) is separable, \(A\) is non-invertible and the essential supremum of
 the multiplicity function of \(A\) does not exceed \(n\);
\item \(A\) is non-invertible and there is a finite subset \(F\) of \(H\) such
 that \([F]_A = H\) and \(\card(F) \leqsl n\).
\end{enumerate}
\end{pro}

Before giving a proof, let us first separate a special case of the above result
that will be applied several times in the sequel:

\begin{lem}{single}
\TFCAE
\begin{enumerate}[\upshape(i)]
\item there exists a pure isometry \(V \in \vvV(A)\) with deficiency index
 \(1\);
\item \(H\) is separable, \(A\) is non-invertible and has simple spectrum.
\end{enumerate}
\end{lem}
\begin{proof}
If \(V \in \vvV(A)\) is pure and has deficiency index \(1\), then
\(\RrR(V)^{\perp}\) is generated by a single unit vector, say \(z\). In
particular, \(A\) is non-invertible (since \(z \notin \RrR(A)\)). Moreover, it
follows from \THM{Wold} that \(H = [z]_A\). So, \(z\) is a cyclic vector for
\(A\) and therefore \(H\) is separable and \(A\) has simple spectrum.\par
To prove the reverse implication, we model \(A\) as the multiplication operator
\(M_{\mu}\) by independent variable on \(L^2(\mu)\) where \(\mu\) is
a probabilistic Borel measure on the spectrum \(K \subset [0,\|A\|]\) of \(A\)
(consult, e.g., Theorem~3.4 in Chapter~IX of \cite{co1}). That is,
\((M_{\mu} f)(t) = t f(t)\) for any \(f \in L^2(\mu)\) and \(t \in K\). Since
\(A\) is one-to-one, \(\mu(\{0\}) = 0\). According to \THM{Wold}, we only need
to show that there is \(u \notin \RrR(M_{\mu})\) such that \([u]_{M_{\mu}} =
L^2(\mu)\). Since \(A\) (that is, \(M_{\mu}\)) is non-invertible, there is \(f
\in L^2(\mu)\) that is not a value of \(M_{\mu}\). This means that \(\int_K
\frac{|f(t)|^2}{t^2} \dint{\mu}(t) = \infty\). By an analogous reasoning, also
\(u \df 1+|f| \in L^2(\mu)\) is not a value of \(M_{\mu}\). It follows from
the description of all (closed linear) invariant subspaces of self-adjoint
operators of the form \(M_{\mu}\) (consult, e.g., Corollary~6.9 in Chapter~IX of
\cite{co1}) that there is a Borel set \(\sigma \subset K\) such that
\([u]_{M_{\mu}} = \{g \in L^2(\mu)\dd\ g = 0\ \textup{\(\mu\)-a.e. on
\(\sigma\)}\}\). But \(u \in [u]_{M_{\mu}}\) and therefore \(\mu(\sigma) = 0\).
Consequently, \([u]_{M_{\mu}} = L^2(\mu)\) and we are done.
\end{proof}

\begin{proof}[Proof of \PRO{fin}]
First of all, note that all items (i)--(iii) imply that \(H\) is separable and
\(A\) is non-invertible. So, everywhere below we assume these two properties:
that \(H\) is separable and \(A\) is non-invertible.\par
Implication ``(i)\(\implies\)(iii)'' is immediate: if \(V \in \vvV(A)\)
witnesses (i), then \(Z \df \RrR(V)^{\perp}\) has an orthonormal basis
consisting of \(n\) vectors, say \(e_1,\ldots,e_n\). Then
\([\{e_1,\ldots,e_n\}]_A = [Z]_A\) and it follows from \THM{Wold} that \([Z]_A =
H\).\par
Now assume (iii) holds. Let \(F\) be as specified therein. We claim that there
are \(k \in \{1,\ldots,n\}\) and unit vectors \(z_1,\ldots,z_k\) such that
\begin{equation}\label{eqn:aux6}
H = \bigoplus_{j=1}^k [z_j]_A.
\end{equation}
To this end, we proceed by induction on \(n\). When \(n = 1\), our conclusion
easily follows. So, assume \(n > 1\), choose any \(a \in F\), put \(F_0 \df F
\setminus \{a\}\), \(H_0 \df [F_0]_A\) and \(A_0 \df A\restriction{H_0} \in
\Bb_+(H_0)\) and apply the induction hypothesis to \(H_0\) and \(A_0\) (and
\(F_0\)): there are \(\ell \in \{1,\ldots,n-1\}\) and unit vectors \(z_1,\ldots,
z_{\ell}\) such that \(H_0 = \bigoplus_{j=1}^{\ell} [z_j]_{A_0}\). If \(H_0 =
H\), just put \(k = \ell\) to finish the proof of \eqref{eqn:aux6}. When \(H_0
\neq H\), proceed as follows. Since \(H = [F]_A\) coincides with the closure of
\([a]_A + [F_0]_A = [a]_A + H_0\), the subspace \(P_{H_0^{\perp}}([a]_A)\) is
dense in \(H_0^{\perp}\). Further, \(P_{H_0^{\perp}}\) commutes with \(A\) and
thus \(P_{H_0^{\perp}}([a]_A)\) is dense in \([P_{H_0^{\perp}}(a)]_A\). So,
\([P_{H_0^{\perp}}(a)]_A = H_0^{\perp}\), \(b \df P_{H_0^{\perp}}(a) \neq 0\)
and it is sufficient to define \(k\) as \(\ell+1\) and \(z_k\) as
\(\frac{b}{\|b\|}\) to get \eqref{eqn:aux6}. Since each of the subspaces
\([z_j]_A\) is invariant under \(A\) (and \(z_j\) is a cyclic vector for
the restriction of \(A\) to \([z_j]_A\)), we see that \(A\) is the direct sum of
at most \(n\) self-adjoint operators with simple spectrum---which yields
(ii).\par
Finally, assume that (ii) is fulfilled. This means that \(A\) is the direct sum
of at most \(n\) self-adjoint operators with simple spectrum, say \(A =
\bigoplus_{j=1}^k A_j\) where \(k \leqsl n\) (and \(A_j\) has simple spectrum).
Then each of \(A_j\) is positive with trivial kernel and one of them, say
\(A_1\), is non-invertible. Using e.g. the spectral measure of \(A_1\), we can
decompose \(A_1\) as \(A_1 = \bigoplus_{m=1}^{\infty} B_m\) where each \(B_m\)
acts on a non-zero Hilbert space and \(\lim_{m\to\infty} \|B_m\| = 0\). Now we
decompose the set of all positive integers as the union of \(n\) pairwise
disjoint sets \(J_1,\ldots,J_n\) in a way such that \(J_1\) is infinite and for
any \(j \in \{2,\ldots,n\}\):
\begin{itemize}
\item if \(A_j\) is non-invertible, then \(J_j = \varempty\);
\item if \(A_j\) is invertible, then \(J_j\) is infinite and \(\|B_s\| <
 1/\|A_j^{-1}\|\) for each \(s \in J_j\).
\end{itemize}
Now define operators \(C_1,\ldots,C_n\) as follows:
\begin{itemize}
\item \(C_1 = \bigoplus_{s \in J_1} B_s\);
\item \(C_j = A_j\) if \(j > 1\) and \(A_j\) is non-invertible;
\item \(C_j = (\bigoplus_{s \in J_j} B_s) \oplus A_j\) in all other cases.
\end{itemize}
It follows from the above construction that:
\begin{itemize}
\item \(A\) is unitarily equivalent to \(\bigoplus_{j=1}^n C_j\);
\item each of \(C_j\) is positive, non-invertible and has trivial kernel;
\item each of \(C_j\) has simple spectrum.
\end{itemize}
Only the last of these properties can be seen as non-trivial, so let us briefly
explain it. Since \(B_j \df \bigoplus_{s \in J_j} B_s\) is the restriction of
\(A_1\) to an invariant subspace of \(A_1\) and \(A_1\) has simple spectrum,
\(B_j\) has simple spectrum as well. Finally, if \(j > 1\) and \(A_j\) is
invertible, then \(\|B_j\| < 1/\|A_j^{-1}\|\) which implies that the spectra
of \(B_j\) and \(A_j\) are disjoint. But then \(B_j \oplus A_j\) has simple
spectrum, as both \(A_j\) and \(B_j\) have so.\par
To conclude the proof, apply \LEM{single} to each of \(C_j\): there is a pure
isometry \(S_j\) with deficiency index \(1\) (acting on an appropriate Hilbert
space) such that \(C_j S_j\) is positive. Then also \((\bigoplus_{j=1}^n C_j)
(\bigoplus_{j=1}^n S_j)\) is positive. So, we complete the proof by noticing
that \(\bigoplus_{j=1}^n S_j\) is a pure isometry with deficiency index \(n\)
and that \(A\) is unitarily equivalent to \(\bigoplus_{j=1}^n C_j\).
\end{proof}

\begin{pro}{aleph}
\TFCAE
\begin{enumerate}[\upshape(i)]
\item there exists a pure isometry \(V \in \vvV(A)\) with deficiency index
 \(\aleph_0\);
\item \(H\) is separable and \(A\) is non-invertible.
\end{enumerate}
\end{pro}
\begin{proof}
The argument is similar to a part of the previous proof and goes as follows. It
is clear that (ii) is implied by (i). Assume (i) holds and let \(B\) be
a maximal set of unit vectors in \(H\) such that \(B \setminus \RrR(A) \neq
\varempty\) and \([b]_A \perp [c]_A\) for distinct \(b, c \in B\). Then \(B\) is
non-empty and (at most) countable, and
\begin{equation}\label{eqn:aux7}
H = \bigoplus_{b \in B} [b]_A.
\end{equation}
We fix \(e \in B\) such that \(e \notin \RrR(A)\). We infer that
\(A_0 \df A\restriction{[e]_A}\) is non-invertible (as an operator in
\(\Bb([e]_A)\)). So, we may decompose \(A_0\) (using, e.g., the spectral measure
of \(A_0\)) as \(A_0 = \bigoplus_{n=1}^{\infty} D_n\) where each \(D_n\) acts on
a non-zero Hilbert space and \(\lim_{n\to\infty} \|D_n\| = 0\). Denote by \(C\)
the set of all \(b \in B\) such that \(A\restriction{[b]_A}\) is invertible (in
\(\Bb([b]_A)\)). We divide the set of all positive integers into pairwise
disjoint sets \(J_b\ (b \in B)\) in a way such that:
\begin{itemize}
\item \(J_e\) is infinite;
\item if \(b \in C\), then \(J_b\) is infinite and \(\|D_s\| <
 1/\|(A\restriction{[b]_A})^{-1}\|\) for all \(s \in J_b\);
\item \(J_b = \varempty\) in all other cases.
\end{itemize}
Finally, decompose \(J_e\) as the union of pairwise disjoint infinite sets
\(I_0,I_1,\ldots\) Now we define operators \(T_b\ (b \in B)\) and \(T^{(n)}\
(n=1,2,\ldots)\) by the rules:
\begin{itemize}
\item \(T_e = \bigoplus_{k \in I_0} D_k\);
\item \(T_b = (\bigoplus_{s \in J_b} D_s) \oplus (A\restriction{[b]_A})\) for
 all \(b \in C\);
\item \(T_b = A\restriction{[b]_A}\) for all \(b \in B \setminus (C \cup
 \{e\})\);
\item \(T^{(n)} = \bigoplus_{k \in I_n} D_k\) for all \(n > 0\).
\end{itemize}
For simplicity, gather all the operators defined above in a sequence \(L_1,
L_2,\ldots\) Since all the sets \(I_n\ (n \geqsl 0)\) and \(J_b\ (b \in B
\setminus \{e\})\) are pairwise disjoint and their union coincides with the set
of all positive integers, one shows that \(A\) is unitarily equivalent to
\(\bigoplus_{n=1}^{\infty} L_n\) (thanks to \eqref{eqn:aux7}). Furthermore, it
follows from the construction that each of \(L_n\) is non-invertible and has
a cyclic vector (cf.\ the proof of \PRO{fin}). So, thanks to \LEM{single} for
any \(n > 0\) there is a pure isometry \(S_n\) with deficiency index \(1\) (that
acts on a suitable Hilbert space) such that \(L_n S_n\) is positive. Then \(S
\df \bigoplus_{n=1}^{\infty} S_n\) is a pure isometry with deficiency index
\(\aleph_0\) and belongs to \(\vvV(\bigoplus_{n=1}^{\infty} L_n)\). This easily
implies that (i) holds.
\end{proof}

\begin{pro}{non-sep}
Let \(\alpha\) be an uncountable cardinal number and \(E\) denote the spectral
measure of \(A\). \TFCAE
\begin{enumerate}[\upshape(i)]
\item there exists a pure isometry \(V \in \vvV(A)\) with deficiency index
 \(\alpha\);
\item \(\dim(H) = \alpha\) and there is \(Z \in \zzZ(A)\) such that \(\dim(Z) =
 \alpha\);
\item \(\dim(H) = \alpha\) and \(\dim \RrR(E((0,\epsi))) = \alpha\) for any
 \(\epsi > 0\).
\end{enumerate}
\end{pro}
\begin{proof}
First of all, note that all conditions specified in the proposition imply that
\begin{equation}\label{eqn:aux8}
\dim(H) = \alpha.
\end{equation}
Thus, everywhere below we assume \eqref{eqn:aux8}. Note also that (ii) is
implied by (i) thanks to \PRO{iso-lin}.\par
Now assume (ii) holds and let \(Z\) be as specified therein. Let \(B\) be
a maximal set of unit vectors in \(Z\) such that \([b]_A \perp [c]_A\) for all
distinct \(b, c \in B\). We claim that
\begin{equation}\label{eqn:aux9}
\card(B) = \alpha.
\end{equation}
To convince oneself of that, assume \(\card(B) < \alpha\). Then also
\(\dim(\bigoplus_{b \in B} [b]_A) < \dim(Z)\) and therefore there exists a unit
vector \(z \in Z\) orthogonal to \(\bigoplus_{b \in B} [b]_A\). This yields that
\([b]_A \perp [z]_A\) for any \(b \in B\) which contradicts the maximality of
\(B\). So, \eqref{eqn:aux9} holds.\par
Now for each \(b \in B\) denote by \(P_b\) the orthogonal projection onto
\([b]_A\). Additionally, for a fixed \(\epsi > 0\) set \(Q \df E((0,\epsi))\).
Note that \(Q\) commutes with all \(P_b\) and that to get (iii) it is sufficient
to show that \(Q P_b \neq 0\) for any \(b \in B\). To this end, fix \(b \in B\)
and assume that, on the contrary, \(Q P_b = 0\). Then \([b]_A \subset
E([\epsi,\infty))\). This means that \(\scalar{Ax}{x}_H \geqsl \epsi \|x\|^2\)
for any \(x \in [b]_A\). Consequently, \(A\restriction{[b]_A}\) is invertible.
But \([b]_A\) is invariant for \(A\) and \(b \notin \RrR(A)\) which contradicts
the invertibility of that restriction.\par
Finally, assume (iii) holds. We want to show that there is \(V\) witnessing (i).
Let \(\{E_s\}_{s \in S}\) be a maximal family of closed linear subspaces such
that:
\begin{itemize}
\item \(E_s\) is separable and \(A(E_s) \varsubsetneq E_s\) for all \(s \in S\);
\item \(E_s \perp E_t\) for any distinct \(s, t \in S\).
\end{itemize}
We claim that
\begin{equation}\label{eqn:aux10}
\card(S) = \alpha.
\end{equation}
To see this, set
\begin{equation}\label{eqn:aux11}
F \df (\bigoplus_{s \in S} E_s)^{\perp}
\end{equation}
and note that \(A(F) \subset F\). Actually, we have \(A(F) = F\). Indeed, if
there was \(z \in F \setminus A(F)\), then \([z]_A\) would be separable and
orthogonal to all \(E_s\) which would contradict the maximality of the family
\(\{E_s\}_{s \in S}\).\par
Further, it follows from the inverse mapping theorem that for some \(\epsi >
0\), \(\scalar{Ax}{x}_H \geqsl \epsi \|x\|^2\) for all \(x \in F\). This
inequality implies that \(E((0,\epsi)) P_F = 0\). Consequently,
\(\RrR(E((0,\epsi))) \subset F^{\perp} = \bigoplus_{s \in S} E_s\). So,
the conclusion of (iii) and the separability of all \(E_s\) yields
\eqref{eqn:aux10}.\par
Further, keeping the setting \eqref{eqn:aux11}, decompose \(F\) as \(F =
\bigoplus_{t \in T} W_t\) where each \(W_t\) is a separable closed linear
subspace of \(F\) invariant under \(A\). Then \(\card(T) \leqsl \alpha\) and
hence there is a one-to-one mapping \(\kappa\dd T \to S\). Now define, for \(s
\in S\), \(H_s\) as follows:
\begin{itemize}
\item \(H_s = E_s\) provided \(s \notin \kappa(T)\);
\item \(H_s = E_s \oplus W_{\kappa^{-1}(s)}\) otherwise.
\end{itemize}
Observe that \(H = \bigoplus_{s \in S} H_s\), and for any \(s \in S\), \(H_s\)
is separable and \(A(H_s) \varsubsetneq H_s\). So, it follows from \PRO{aleph}
that for any \(s \in S\) there is a pure isometry \(V_s\dd H_s \to H_s\) that
belongs to \(\vvV(A\restriction{H_s})\). To complete the proof, define \(V \in
\vvV(A)\) by \(V = \bigoplus_{s \in S} V_s\) and note that \(V\) is a pure
isometry with deficiency index \(\alpha\) (by \eqref{eqn:aux10}).
\end{proof}

Observe that item (ii) of the above theorem---as well as the collection
\(\zzZ(A)\)---depends only on the range of the operator \(A\). Further
conditions (also formulated only in terms of operator ranges) equivalent to this
item are a subject of \THM{pure} from the next section.

\begin{proof}[Proof of \THM{main}]
Just observe that the equivalence of conditions (i) and (ii) immediately follows
from \PRO[s]{fin}, \ref{pro:aleph} and \ref{pro:non-sep}, whereas the remaining
(additional) part of the theorem is a reformulation of \LEM{single}.
\end{proof}

\begin{rem}{simple}
It was shown earlier---in \cite{g-p}---that one-to-one Hankel operators
satisfying double positivity condition have simple spectra. Both the proofs---in
the paper cited above and ours---are based on the same idea, which is a kind of
folklore in operator theory. For the details, consult \cite{kat} or
Proposition~2.5 in \cite{g-p} together with the preceding paragraph (therein).
\end{rem}

The following result is a simple consequence of a deep theorem from \cite{mpt}.
Here we give its brief proof.

\begin{cor}{Hankel}
The essential supremum of the multiplicity function of a positive Hankel
operator \(A\) with trivial kernel does not exceed \(2\).
\end{cor}
\begin{proof}
If \(S\) denotes the (classical) unilateral shift, then \(S^2\) is a pure
isometry with deficiency index \(2\) and \(A S^2\) is a positive operator
(because \(A\) is positive and Hankel). So, the assertion follows from
\THM{main} (for \(\mM = 2\)).
\end{proof}

In \cite{g-p} the authors showed that for any bounded Hankel operator \(A\)
satisfying double positivity condition the operator
\(A\restriction{\overline{\RrR}(A)}\) has simple spectrum. It is also well-known
(and easy to prove) that the kernel of a Hankel operator is either trivial or
infinite-dimensional. So, the following problem naturally arises:

\begin{cnj}{Hankel}
Let \(A\) be a bounded positive operator on a separable Hilbert space such that
\(\NnN(A)\) is infinite-dimensional, \(\RrR(A)\) is non-closed and
\(A\restriction{\overline{\RrR}(A)}\) has simple spectrum. Then \(A\) is
unitarily equivalent to a Hankel operator satisfying double positivity
condition.
\end{cnj}

\section{Operator ranges}

In this part we give further conditions on a dense operator range \(\RrR\)
contained in \(H\) equivalent to the existence of a closed linear subspace \(Z\)
of \(H\) such that \(Z \cap \RrR = \{0\}\) and \(\dim(Z) = \dim(H)\) (cf.\ item
(ii) in \PRO{non-sep}). To this end we need to recall some well-known facts
about operator ranges.\par
A linear subspace \(\RrR\) of a Hilbert space \(H\) is called
an \textit{operator range} if there exist a Hilbert space \(K\) and a bounded
operator \(T\dd K \to H\) such that \(\RrR(T) = \RrR\). The following is a basic
result on operator ranges (see Theorem~1.1 in \cite{f-w}).

\begin{thm}{basic}
For any operator range \(\RrR\) in \(H\) there are mutually orthogonal closed
linear subspaces \(H_1,H_2,\ldots\) such that
\begin{equation}\label{eqn:opr}
\RrR = \OPN{ran}[H_1,H_2,\ldots] \df \Bigl\{\oplus_{n=1}^{\infty} x_n \in
\bigoplus_{n=1}^{\infty} H_n\dd\ \sum_{n=1}^{\infty} (2^n \|x_n\|)^2 <
\infty\Bigr\}.
\end{equation}
Moreover, the operator range \(\RrR\) given by \eqref{eqn:opr} is dense in \(H\)
iff \(H = \bigoplus_{n=1}^{\infty} H_n\).
\end{thm}

It is worth noting that for a given operator range \(\RrR\) a sequence \(H_1,
H_2,\ldots\) such that \(\RrR = \OPN{ran}[H_1,H_2,\ldots]\) (cf.\
\eqref{eqn:opr}) is, in general, not unique. Also, some of the spaces \(H_n\)
can be zero-dimensional. For the purposes of this section, let us introduce
the following

\begin{dfn}{type}
Let \(\alpha_1,\alpha_2,\ldots\) be any sequence of cardinal numbers. We say
an operator range \(\RrR\) in \(H\) is \textit{of type
\(\ueR(\alpha_n)_{n=1}^{\infty}\)} if there are mutually orthogonal closed
linear subspaces \(H_1,H_2,\ldots\) of \(H\) such that \(\RrR =
\OPN{ran}[H_1,H_2,\ldots]\) (cf.\ \eqref{eqn:opr}) and \(\dim(H_n) = \alpha_n\)
for any \(n > 0\).\par
Two operator ranges \(\RrR_1\) in \(H_1\) and \(\RrR_2\) in \(H_2\) are
\textit{equivalent} if there is a unitary operator \(U\dd H_1 \to H_2\) such
that \(U(\RrR_1) = \RrR_2\).
\end{dfn}

A single operator range can be of many different types. However, it is easy to
see that \textbf{dense} operator ranges of a given type are all equivalent. When
two types represent equivalent dense operator ranges is a subject of Theorem~3.3
in \cite{f-w}. Here we skip the details.\par
The main result of this section is

\begin{thm}{pure}
Let \(\RrR\) be a dense operator range of type \(\ueR(\alpha_n)_{n=1}^{\infty}\)
in a Hilbert space \(H\) of dimension \(\alpha \geqsl \aleph_0\). \TFCAE
\begin{enumerate}[\upshape(a)]
\item there exists a closed linear subspace \(Z\) of \(H\) such that \(Z \cap
 \RrR = \{0\}\) and \(\dim(Z) = \alpha\);
\item for any \(\beta < \alpha\) there exists a closed linear subspace \(Y\) of
 \(H\) such that \(Y \cap \RrR = \{0\}\) and \(\dim(Y) > \beta\);
\item for any \(\beta < \alpha\) and \(n > 0\) there is \(m > n\) such that
 \(\sum_{j=n+1}^m \alpha_j > \beta\);
\item \(\dim(K^{\perp}) = \alpha\) for any linear subspace \(K\) of \(\RrR\)
 that is closed in \(H\);
\item there exists a dense operator range \(\RrR'\) in \(H\) such that \(\RrR'
 \cap \RrR = \{0\}\);
\item there exists \(U \in \Uu(H)\) for which \(U(\RrR) \cap \RrR = \{0\}\);
\item there exists a closed linear subspace \(W\) of \(H\) such that \(W \cap
 \RrR = W^{\perp} \cap \RrR = \{0\}\).
\end{enumerate}
Moreover, \(\dim(W) = \dim(W^{\perp}) = \alpha\) for any \(W\) witnessing
\textup{(g)}.
\end{thm}

Before giving a proof, first we comment on the assertion of (f) in the above
result: in the case when \(\RrR\) is a non-closed operator range in
a \textbf{separable} Hilbert space \(H\) it was first proved by Dixmier
\cite{dix} (consult also Theorem~3.6 in \cite{f-w}). He started his proof with
a specific example of an operator range for which (f) holds and then described
an elegant technique to get the assertion in full generality (in the separable
case). Here we will generalize his method, but instead of using his ``starting''
example we propose a new approach to this issue---our starting tool will be
the following proposition, which may be already known, but we could not find it
in the literature. This result can be considered interesting in its own right.

\begin{pro}{comp}
Let \(C\dd K \to H\) be a compact operator from a Hilbert space \(K\) into
a \textbf{separable} infinite-dimensional Hilbert space \(H\). Then the set
\(\Delta\) of all \(U \in \Uu(H)\) such that \(U(\RrR(C)) \cap \RrR(C) = \{0\}\)
is dense \(\ggG_{\delta}\) in the Polish group \(\Uu(H)\) equipped with
the strong operator topology.
\end{pro}
\begin{proof}
First of all, recall that for separable \(H\) the space \(\Uu(H)\) is separable
and completely metrizable in the strong operator topology. Moreover, for any
Hilbert space \(K\), \(\Uu(K)\) is a topological group with respect to this
topology. So, in the separable case we can apply Baire's theorem, which will
lead us to the final conclusion after showing the following property, valid in
all (that is, possibly non-separable) infinite-dimensional Hilbert spaces \(H\):
\begin{itemize}
\item[\((\star)\)]\itshape
 For any norm compact set \(M \subset H\) disjoint from the origin, the set
 \(\Omega\) of all \(U \in \Uu(H)\) such that \(U(M) \cap M = \varempty\) is
 open and dense in the topological group \(\Uu(H)\) equipped with the strong
 operator topology.
\end{itemize}
First we prove \((\star)\), then we will show how it leads to the whole
conclusion of the proposition.\par
We may and do assume that \(M \neq \varempty\). It easily follows from
the compactness of \(M\) that \(\Omega\) is open in \(\Uu(H)\). Indeed, fix \(U
\in \Omega\) and take \(\delta > 0\) such that \(\|Ux - y\| \geqsl 4\delta\) for
all \(x, y \in M\). Further, let \(F \subset M\) be a finite non-empty
\(\delta\)-net in \(M\). It is then easy to check that \(\|Vx - y\| \geqsl
\delta\) for all \(x, y \in M\) and any \(V \in \Uu(H)\) such that \(\|Vf - Uf\|
< \delta\) for any \(f \in F\). Consequently, each such \(V\) belongs to
\(\Omega\) and hence \(\Omega\) is open in the strong operator topology. To show
that \(\Omega\) is dense in \(\Uu(H)\), we need to know that \(H\) is
infinite-dimensional. Fix \(U \in \Uu(H)\), a finite non-empty set \(F\) in
\(H\) and \(\epsi > 0\). Our aim is to show that there is \(V \in \Omega\) such
that \(\|Vf - Uf\| \leqsl \epsi\) for any \(f \in F\). We may and do assume that
\(0 \notin F\). For simplicity, put \(L \df M \cup F\) and take \(m > 0\) such
that
\begin{equation}\label{eqn:aux14}
\|y\| \geqsl m \qquad (y \in L).
\end{equation}
Next, choose any \(\delta \in (0,1)\) satisfying
\begin{equation}\label{eqn:aux15}
2 \|f\| \delta < \epsi \qquad (f \in F)
\end{equation}
and let \(S \subset L\) be a finite \((m\delta/4)\)-net in \(L\) containing
\(F\). Set \(E \df \lin(S)\) and take any \(W \in \Uu(H)\) such that
\begin{equation}\label{eqn:aux16}
W(E) \perp U(E)+E.
\end{equation}
We conclude from the above property (and the fact that \(E\) is
finite-dimensional) that there exists \(V \in \Uu(H)\) that satisfies
\begin{equation}\label{eqn:aux17}
Vx = \delta Wx + \sqrt{1-\delta^2} Ux \qquad (x \in E).
\end{equation}
Since \(F \subset E\), it follows from \eqref{eqn:aux16} and \eqref{eqn:aux17}
that for any \(f \in F\),
\[\|Vf - Uf\|^2 = 2(1 - \sqrt{1-\delta^2}) \|f\|^2 \leqsl 2 \delta^2 \|f\|^2 <
\epsi^2\]
(where the last inequality is a consequence of \eqref{eqn:aux15}). Thus, to end
the proof of \((\star)\), it remains to check that \(V(M) \cap M = \varempty\).
To this end, first take \(a, b \in S \subset E \cap L\). Again, we infer from
\eqref{eqn:aux16} and \eqref{eqn:aux17} that \(\|Va - b\|^2 = \|a\|^2 + \|b\|^2
- 2 \sqrt{1-\delta^2} \RE\scalar{Ua}{b}_H\), so
\[\|Va - b\|^2 \geqsl \|a\|^2 + \|b\|^2 - 2 \sqrt{1-\delta^2} \|a\| \cdot \|b\|
\geqsl \delta^2 \|b\|^2 \geqsl \delta^2 m^2,\]
by \eqref{eqn:aux14}. Now if \(x, y \in M\) are arbitrary, choose \(a, b \in S\)
such that \(\|x - a\| \leqsl m\delta/4\) and \(\|y - b\| \leqsl m\delta/4\). We
then have \(\|Vx - y\| \geqsl \|Va - b\| - \|Va - Vx\| - \|y - b\| \geqsl
m\delta/2\) and we are done.\par
Having \((\star)\), the assertion of the proposition can briefly be proven.
Since \(C\) is compact and the closed unit ball \(\bar{B}_K\) in \(K\) is weakly
compact, the set \(D \df C(\bar{B}_K)\) is norm compact. Observe that \(\RrR(C)
\setminus \{0\} = \bigcup_{n=1}^{\infty} M_n\) where \(M_n = nC \setminus
(\frac1n B_H)\) where \(B_H\) is the open unit ball in \(H\). We infer from
\((\star)\) that the set \(\Omega_n \df \{U \in \Uu(H)\dd\ U(M_n) \cap M_n =
\varempty\}\) is open and dense in \(\Uu(H)\). Finally, Baire's theorem yields
that the intersection of all \(\Omega_n\), which coincides with \(\Delta\), is
dense in \(\Uu(H)\).
\end{proof}

In the proof of \THM{pure} we shall also apply the following result.

\begin{cor}{perp}
There exists a dense operator range \(\RrR_0\) in a separable Hilbert space
\(K\) and a closed linear subspace \(W\) of \(K\) such that \(W \cap \RrR_0 =
W^{\perp} \cap \RrR_0 = \{0\}\).
\end{cor}
\begin{proof}
Let \(K \df H \oplus H\) (where \(H\) is infinite-dimensional and separable) and
\(W \df H \oplus \{0\} \subset K\). Take any compact self-adjoint operator
\(A\dd H \to H\) with trivial kernel and choose---applying \PRO{comp}---any \(U
\in \Uu(H)\) such that
\begin{equation}\label{eqn:aux18}
U(\RrR(A)) \cap \RrR(A) = \{0\}.
\end{equation}
Let \(B\dd H \to K\) be given by \(Bx \df Ax \oplus UAU^{-1}x\) and set \(\RrR_0
\df \RrR(B)\). Since both \(A\) and \(U\) are one-to-one, we have \(W \cap
\RrR_0 = W^{\perp} \cap \RrR_0 = \{0\}\). So, we only need to show that
\(\RrR_0\) is dense in \(K\), which simply follows from \eqref{eqn:aux18}: if
\(a \oplus b\) is orthogonal to \(\RrR_0\), then \(Aa + UAU^{-1}b = 0\) and,
consequently, \(a = b = 0\).
\end{proof}

\begin{proof}[Proof of \THM{pure}]
For the aim of this proof, take a sequence of mutually orthogonal closed linear
subspaces \(H_1,H_2,\ldots\) of \(H\) such that
\begin{equation}\label{eqn:aux19}
\RrR = \OPN{ran}[H_1,H_2,\ldots]
\end{equation}
and \(\dim(H_n) = \alpha_n\) for any \(n > 0\). Since \(\RrR\) is dense, we have
(by \THM{basic}):
\begin{equation}\label{eqn:aux12}
H = \bigoplus_{n=1}^{\infty} H_n.
\end{equation}
Before passing to the main proof, consider an additional condition:
\begin{itemize}
\item[(h)] for any \(\beta < \alpha\) and \(n > 0\) there exists a closed linear
 subspace \(K\) of \(H\) such that
 \begin{equation}\label{eqn:aux13}
 K \cap \Bigl(\bigoplus_{j=1}^n H_j\Bigr) = \{0\} \qquad \textup{and} \qquad
 \dim(K) > \beta.
 \end{equation}
\end{itemize}
(Note that (h) is a weakening of each of (a), (b), (c), (d) and (g).) We will
show that (h) is equivalent to each of (a)--(g). For the reader's convenience,
let us draw the scheme of the proof:
(a)\(\implies\)(b)\(\implies\)(d)\(\implies\)(h)\(\implies\)(c)\(\implies\)(a);
then (f)\(\implies\)(e)\(\implies\)(h) and (g)\(\implies\)(a) (together with
the additional claim of the theorem), and finally (c)\(\implies\)(f),(g).\par
Of couse, (a) implies (b). If \(\beta < \alpha\) and \(Y\) witnesses (b) (for
\(\beta\)), and \(K\) is as specified in (d), then \(P_{K^{\perp}}\) is
one-to-one on \(Y\). Consequently, \(\dim(K^{\perp}) = \dim(\RrR(P_{K^{\perp}}))
\geqsl \dim(Y) > \beta\) which yields (d). It is obvious that (d) implies (h).
Now assume (h) holds and fix \(\beta < \alpha\) and \(n \geqsl 0\). For
simplicity, denote by \(P\) the orthogonal projection onto the orthogonal
complement (in \(H\)) of \(\bigoplus_{j=1}^n H_j\). By (h), there is a closed
linear subspace \(K\) of \(H\) that satisfies \eqref{eqn:aux13}. As argued
previously, we conclude that \(\dim(\RrR(P)) \geqsl \dim(K) > \beta\). But
\(\RrR(P) = \bigoplus_{j=n+1}^{\infty} H_j\) (thanks to \eqref{eqn:aux12}) and
hence \(\sum_{j>n} \alpha_j > \beta\). So, one can find \(m > n\) such that
\(\sum_{j=n+1}^m \alpha_j > \beta\) which gives (c). Finally, assume (c) holds.
Consider a bounded operator \(A\dd H \to H\) defined as follows:
\[A(\oplus_{n=1}^{\infty} x_n) \df \oplus_{n=1}^{\infty} 2^{-n} x_n, \qquad
\oplus_{n=1}^{\infty} x_n \in \bigoplus_{n=1}^{\infty} H_n.\] 
It readily follows from \eqref{eqn:aux19} that \(\RrR(A) = \RrR\). Moreover, it
is also easy to show that for arbitrarily fixed \(\epsi > 0\),
\(\bigoplus_{n=j}^{\infty} H_j \subset \RrR(E((0,\epsi)))\) for sufficiently
large \(n > 0\) where \(E\) is the spectral measure of \(A\). Consequently, we
infer from (c) that \(\dim(\RrR(E((0,\epsi)))) = \alpha\) and it suffices to
apply \PRO{non-sep} to get (a).\par
Further, (e) is easily implied by (f) as \(U(\RrR)\) (for any \(U \in \Uu(H)\))
is a dense operator range in \(H\). And if (e) holds, \(\RrR' =
\OPN{ran}[H_1',H_2',\ldots]\) for suitable sequence \(H_1',H_2',\ldots\) of
mutually orthogonal closed linear subspaces of \(H\) such that \(H =
\bigoplus_{n=1}^{\infty} H_n'\). Then \(\sum_{n=1}^{\infty} \dim(H_n') =
\alpha\) and \((\bigoplus_{j=1}^m H_j') \cap (\bigoplus_{k=1}^n H_k) = \{0\}\)
for all \(n, m > 0\). These properties easily yield (h).\par
To show the additional claim of the theorem and that (a) follows from (g), it is
sufficient to prove that \(\dim(W) = \alpha\) whenever \(W\) is as specified in
(g). To this end, assume---on the contrary---that \(\dim(W) < \alpha\). Then, by
\eqref{eqn:aux12}, we can find \(n > 0\) such that \(\sum_{j=1}^n \dim(H_j) >
\dim(W)\). This inequality implies that \((\bigoplus_{j=1}^n H_j) \cap W^{\perp}
\neq \{0\}\) which contradicts (g) as \(\bigoplus_{j=1}^n H_j \subset
\RrR\).\par
We turn to the hardest part of the proof---namely, that both (f) and (g) follow
from (c). We adapt Dixmier's proof \cite{dix} (see also Theorem~3.6 in
\cite{f-w}) of the result mentioned in the paragraph following the statement of
\THM{pure} above, but instead of his specific example of a dense operator range
in a separable Hilbert space that satisfies (f) we apply our \PRO{comp} and
\COR{perp}.\par
To simplify further arguments, let us call a linear subspace \(\RrR \subset H\)
of an arbitrary Hilbert space (f,g)-\textit{valid} if both (f) and (g) hold for
\(\RrR\). Here we do not assume that \(\RrR\) is an operator range. In a similar
manner we define (f)-\textit{valid} and (g)-\textit{valid} linear subspaces of
Hilbert spaces. Moreover, for any non-empty set \(J\) and \(\RrR \subset H\) let
\(\RrR^J \subset \bigoplus_{j \in J} H_j\) (with \(H_j = H\) for all \(j \in
J\)) stand for the set of all \(\oplus_{j \in J} x_j\) with \(x_j \in \RrR\) for
any \(j \in J\). Finally, for any infinite cardinal \(\gamma\) we call
an operator range \(\RrR\) of type \(\ueR_{\gamma}\) if it is of type
\(\ueR(\gamma_n)_{n=1}^{\infty}\) where \(\gamma_n = \gamma\) for each \(n\). We
divide the remaining part of the proof of the theorem into the following steps:
\begin{enumerate}[(I)]
\item There are dense operator ranges \(\RrR_f\) and \(\RrR_g\) in separable
 Hilbert spaces such that \(\RrR_f\) is (f)-valid and \(\RrR_g\) is (g)-valid.
\item If \(\RrR\) is (f,g)-valid, each linear subspace of \(\RrR\) is
 (f,g)-valid as well.
\item If \(\RrR\) is (f)-valid (resp.\ (g)-valid), so is \(\RrR^J\) for any set
 \(J \neq \varempty\).
\item If \(\RrR\) is a dense operator range in a separable Hilbert space \(H\)
 and \(J\) is an infinite set, then \(\RrR^J\) contains an operator range of
 type \(\ueR_{\card(J)}\) that is dense in \(H^J\).
\item For any infinite cardinal \(\gamma\), all dense operator ranges of type
 \(\ueR_{\gamma}\) are (f,g)-valid.
\item If \(\RrR\) and \(H\) are as specified in the statement of the theorem and
 (c) is fulfilled, then \(\RrR\) is contained in a dense (in \(H\)) operator
 range of type \(\ueR_{\alpha}\). In particular, \(\RrR\) is (f,g)-valid.
\end{enumerate}
Note that property (VI) is exactly what we want. Below we give brief proofs of
the above items (I)--(VI).\par
Property (I) is covered by \PRO{comp} and \COR{perp}; (II) is obvious; whereas
(III) follows from a simple argument on direct sums: if \(U\) is a unitary
operator on \(H\) such that \(U(\RrR) \cap \RrR = \{0\}\), then \(U^J \df
\bigoplus_{j \in J} U_j\) (with \(U_j = U\) for all \(j \in J\)) is a unitary
operator on \(H^J\) such that \(U^J(\RrR^J) \cap \RrR^J = \varempty\)
(similarly: if \(W \subset H\) witnesses (g) for \(\RrR\), then \(W^J\)
witnesses (g) for \(\RrR^J\)).\par
We turn to (IV). Assume \(\RrR\) is as specified therein. We may and do assume
that \(\RrR \neq H\). Take a sequence of mutually orthogonal closed linear
subspaces \(H_1,H_2,\ldots\) of \(H\) such that \(\RrR =
\OPN{ran}[H_1,H_2,\ldots]\) and \eqref{eqn:aux12} holds. Since
\(\OPN{ran}[H_1,H_2,\ldots] = \OPN{ran}[\{0\},H_1,H_2,\ldots]\), we may and do
assume that \(H_1 \neq \{0\}\). Finally, there is a sequence of natural numbers
\(1 = \nu_1 < \nu_2 < \ldots\) such that
\begin{equation}\label{eqn:aux20}
H_{\nu_n} \neq \{0\} \qquad (n > 0)
\end{equation}
(because \(\RrR \neq H\)). We consider each of \(H_j = H\) in \(H^J =
\bigoplus_{j \in J} H_j\) as \(H_j = \bigoplus_{n=1}^{\infty} H_{j,n}\) where
\(H_{j,n} = H_n\) for any \(j \in J\) and \(n > 0\). Let \[W\dd H^J =
\bigoplus_{j \in J} \Bigl(\bigoplus_{n=1}^{\infty} H_{j,n}\Bigr) \to
\bigoplus_{n=1}^{\infty} \Bigl(\bigoplus_{j \in J} H_{j,n}\Bigr)\] be
the natural unitary operator that shuffles coordinates:
\(W ((x_{j,n})_{n=1}^{\infty})_{j \in J} =
((x_{j,n})_{j \in J})_{n=1}^{\infty}\). For simplicity, set \(K_n \df
\bigoplus_{j \in J} H_{j,n}\ (n > 0)\) and \(K \df \bigoplus_{n=1}^{\infty}
K_n\). Observe that \(W(\RrR^J)\) contains the space \(\TtT \df
\OPN{ran}[K_1,K_2,\ldots]\). So, to complete the proof of (IV), it is enough to
show that \(\TtT\) contains a dense (in \(K\)) operator range of type
\(\ueR_{\gamma}\) where \(\gamma = \card(J) \geqsl \aleph_0\). To this end, note
that \(\dim(K_{\nu_n}) = \gamma\) for all \(n > 0\) (thanks to
\eqref{eqn:aux20}). Hence, each of the spaces \(K_{\nu_n}\) can be decomposed as
\(K_{\nu_n} = \bigoplus_{k=\nu_n}^{\nu_{n+1}-1} Y_k\) where \(\dim(Y_k) =
\gamma\) for any \(k > 0\) (recall that \(\nu_1 = 1\)). Now relations \(K =
\bigoplus_{k=1}^{\infty} Y_k\) and \(\OPN{ran}[Y_1,Y_2,\ldots] \subset \TtT\)
finish the proof of (IV).\par
As (V) is an immediate consequence of (I)--(IV) and of the fact that all dense
operator ranges of the same type are equivalent, it remains to show (VI). To
this end, we start from \eqref{eqn:aux19} (with \(\dim(H_n) = \alpha_n\)) and
\eqref{eqn:aux12}. It follows from condition (b) that there is an infinite
set \(\Lambda \subset \{1,2,\ldots\}\) such that for any its infinite subset
\(\Lambda_0\) one has \(\sum_{n\in\Lambda_0} \alpha_n = \alpha\). We divide
\(\Lambda\) into pairwise disjoint infinite sets \(\Lambda_1,\Lambda_2,\ldots\)
such that
\begin{equation}\label{eqn:aux21}
\min(\Lambda_n) \geqsl n \qquad (n > 0).
\end{equation}
(We assume that \(\Lambda = \bigcup_{n=1}^{\infty} \Lambda_n\).) Finally, for
each positive \(n\) define \(H_n'\) as follows:
\begin{itemize}
\item \(H_n' = \bigoplus_{k\in\Lambda_n} H_k\) if \(n \in \Lambda\);
\item \(H_n' = (\bigoplus_{k\in\Lambda_n} H_k) \oplus H_n\) otherwise.
\end{itemize}
Note that \(H = \bigoplus_{n=1}^{\infty} H_n'\) (by \eqref{eqn:aux12}) and
\(\dim(H_n') = \alpha\) for all \(n\). What is more, \eqref{eqn:aux21} implies
that \(\OPN{ran}[H_1,H_2,\ldots] \subset \OPN{ran}[H_1',H_2',\ldots]\) which
shows that \(\RrR\) is contained in an operator range of type \(\ueR_{\alpha}\).
The final claim of (VI) is now a consequence of (V) and (II).
\end{proof}

In a similar manner one shows the following result whose proof is left to
the reader.

\begin{pro}{card}
Let \(\RrR\) be a dense operator range of type \(\ueR(\alpha_n)_{n=1}^{\infty}\)
in a Hilbert space \(H\) and let \(\beta\) be an arbitrary cardinal number.
\TFCAE
\begin{enumerate}[\upshape(a)]
\item there exists a closed linear subspace \(Z\) of \(H\) such that \(Z \cap
 \RrR = \{0\}\) and \(\dim(Z) = \beta\);
\item for any cardinal \(\gamma < \beta\) there is a closed linear subspace
 \(W\) of \(H\) such that \(W \cap \RrR = \{0\}\) and \(\dim(W) > \gamma\);
\item \(\sum_{k=n}^{\infty} \alpha_k \geqsl \beta\) for any \(n > 0\).
\item \(\dim(K^{\perp}) \geqsl \beta\) for any linear subspace \(K\) of \(\RrR\)
 that is closed in \(H\).
\end{enumerate}
\end{pro}

Recall that the structure (that is, unitary equivalence type) of the isometry
\(W_Z\) for any \(Z \in \zzZ(A)\) where \(A \in \Bb_+(H)\) is completely
determined by the cardinals \(\alpha(Z) = \dim(Z)\) and \(\beta_A(Z) =
\dim([Z]_A^{\perp})\) (see \THM{Wold}). So, the set \[\Xi(A) \df \{(\dim(Z),
\dim([Z]_A^{\perp}))\dd\ Z \in \zzZ(A)\}\] contains full information on possible
unitary equivalence types of members of \(\vvV(A)\). Our nearest aim is to
determine \(\Xi(A)\). To this end, let us introduce three additional
characteristic quantities:
\begin{itemize}
\item \(\OPN{cdr}(A) \df \min\bigl\{\dim(K^{\perp})\dd\ K \subset \RrR(A)
 \textup{ is a closed linear subspace of } H\bigr\}\);
\item \(\OPN{eig}(A) \df \dim(\OPN{Eig}(A))+1\) if \(\OPN{Eig}(A)\) is
 finite-dimensional and \(\OPN{eig}(A) = \infty\) otherwise where
 \(\OPN{Eig}(A) \df \bigoplus_{\alpha\in\CCC} \NnN(A-\alpha I)\);
\item \(\OPN{ess}(A)\) is a natural number \(n\) if \(A\) acts on a separable
 Hilbert space and \(n\) is the essential supremum of the multiplicity function
 of \(A\), otherwise (that is, if for no \(n\) the previous condition is
 fulfilled) \(\OPN{ess}(A) = \aleph_0\).
\end{itemize}
It is worth noting here that \(\OPN{cdr}(A) = \min_{n>0} \sum_{k=n}^{\infty}
\alpha_k\) if \(\RrR(A)\) is of type \(\ueR(\alpha_n)_{n=1}^{\infty}\), which
simply follows from \PRO{card}.\par
Everywhere below, \(\alpha\) and \(\beta\) are cardinal numbers and ``\(\beta <
\infty\)'' means that \(\beta\) is finite.

\begin{thm}{card}
Let \(A \in \Bb_+(H)\) be non-invertible.
\begin{enumerate}[\upshape(I)]
\item If \(H\) is separable, then \(\Xi(A)\) consists of all pairs
 \((\alpha,\beta)\) such that
 \begin{itemize}
 \item \(\alpha \leqsl \aleph_0\) and \(\beta = \aleph_0\); or
 \item \(\OPN{ess}(A) \leqsl \alpha \leqsl \aleph_0\) and \(\beta <
  \OPN{eig}(A)\).
 \end{itemize}
\item If \(H\) is non-separable, then \(\Xi(A)\) consists of all pairs
 \((\alpha,\beta)\) such that
 \begin{itemize}
 \item \(\alpha < \dim(H)\) and \(\alpha \leqsl \OPN{cdr}(A)\) and \(\beta =
  \dim(H)\); or
 \item \(\alpha = \dim(H) = \OPN{cdr}(A)\) and \(\aleph_0 \leqsl \beta \leqsl
  \dim(H)\); or
 \item \(\alpha = \dim(H) = \OPN{cdr}(A)\) and \(\beta < \OPN{eig}(A)\).
 \end{itemize}
\end{enumerate}
\end{thm}
\begin{proof}
First assume \(H\) is separable. The proof of \PRO{aleph} shows that \(A\) is
a direct sum of two non-invertible positive operarors, say \(A = B \oplus C\).
Since \(B\) is non-invertible, it follows that there is a closed
infinite-dimensional linear subspace \(Z\) of \(\bar{R}(B)\) such that \(Z \cap
\RrR(B) = \{0\}\). Then all closed subspaces \(W\) of \(Z\) belong to
\(\zzZ(A)\) and satisfy \(\dim([W]_A^{\perp}) = \aleph_0\) (since \(C\) is
non-invertible and one-to-one, its domain is infinite-dimensional). So, to
finish the proof in the separable case, it is sufficient to characterize all
pairs \((\alpha,\beta) \in \Xi(A)\) with finite \(\beta\). Observe that if
\(Z \in \zzZ(A)\) is such that \(\dim(Z) = \alpha\) and \(\dim([Z]_A^{\perp}) =
\beta < \aleph_0\), then \([Z]_A^{\perp} \subset \OPN{Eig}(A)\) (recall that
\(\OPN{Eig}(A) = \bigoplus_{\alpha\in\CCC} \NnN(A-\alpha I)\)), because
\([Z]_A^{\perp}\) is a finite-dimensional invariant subspace for \(A\).
Consequently, \(\beta < \OPN{eig}(A)\). Moreover,
\(\OPN{ess}(A\restriction{[Z]_A}) \leqsl \OPN{ess}(A)\) and \PRO{fin} implies
that \(\OPN{ess}(A\restriction{[Z]_A}) \leqsl \alpha\). Conversely, if
\(\OPN{ess}(A) \leqsl \alpha \leqsl \aleph_0\) and \(\beta < \OPN{eig}(A)\),
then there is a linear subspace \(E\) of \(\OPN{Eig}(A)\) of (finite) dimension
\(\beta\) that is invariant for \(A\). Denoting by \(K\) the orthogonal
complement of \(E\), we see that \(K\) is invariant for \(A\) and
\(\OPN{ess}(A\restriction{K}) \leqsl \OPN{ess}(A)\). So,
\(\OPN{ess}(A\restriction{K}) \leqsl \alpha\) and therefore---thanks to
\PRO[s]{fin} and \ref{pro:aleph}---there is \(Z \in \zzZ(A\restriction{K})\)
such that \(\dim(Z) = \alpha\) and \([Z]_A = K\). Then \(\dim([Z]_A^{\perp}) =
\dim(E) = \beta\) and hence \((\alpha,\beta) \in \Xi(A)\).\par
Now we turn to the non-separable case. First assume \((\alpha,\beta) \in
\Xi(A)\) and choose \(Z \in \zzZ(A)\) for which \(\dim(Z) = \alpha\) and \(\beta
= \dim([Z]_A^{\perp})\). It follows from the definition of \(\OPN{cdr}(A)\) that
\(\dim(Z) \leqsl \OPN{cdr}(A)\). If \(\alpha < \dim(H)\), then \(\dim([Z]_A)
\leqsl \max(\alpha,\aleph_0) < \dim(H)\) and hence \(\beta = \dim(H)\). In
the remaining case \(\alpha = \dim(H) = \OPN{cdr}(A)\) and \(\beta\) is either
infinite or \([Z]_A^{\perp} \subset \OPN{Eig}(A)\) and thus finite \(\beta\)
must satisfy \(\beta < \OPN{eig}(A)\). This shows that the condition specified
in (II) is necessary. Now we show its sufficiency. To this end, we fix a pair
\(\alpha,\beta\) of cardinals such that \(\alpha \leqsl \OPN{cdr}(A)\). It
follows from \PRO{card} that there is \(Z \in \zzZ(A)\) with \(\dim(Z) =
\alpha\).\par
First assume \(\alpha < \dim(H)\) (then \(\beta = \dim(H)\). Choose any \(Z \in
\zzZ(A)\) such that \(\dim(Z) = \alpha\). Then \(\dim([Z]_A) \leqsl
\max(\alpha,\aleph_0) < \dim(H)\) and thus \(\dim([Z]_A^{\perp}) = \dim(H) =
\beta\), which implies \((\alpha,\beta) \in \Xi(A)\).\par
Now assume \(\alpha = \dim(H) = \OPN{cdr}(A)\). The proof of \PRO{non-sep} shows
that \(A\) can be decomposed as \(A = \bigoplus_{s \in S} A_s\) where \(\card(S)
= \alpha\) and each of \(A_s\) is non-invertible and acts on a separable Hilbert
space. Moreover, we have shown there that the existence of such a decomposition
is sufficient for the existence of \(Z \in \zzZ(A)\) such that \(\dim(Z) =
\alpha\) and \([Z]_A = H\). We will use this property below.\par
If \(\beta\) is infinite, we take disjoint subsets \(S_1\) and \(S_2\) of \(S\)
such that \(S = S_1 \cup S_2\), \(\card(S_1) = \alpha\) and \(\card(S_2) =
\beta\). We infer from the property evoked above that there is \(W \in
\zzZ(\bigoplus_{s \in S_1} A_s)\) such that \(\dim(W) = \alpha\) and \([W]_A\)
coincides with the domain of \(\bigoplus_{s \in S_1} A_s\). Then
\(\dim([W]_A^{\perp}) = \card(S_2)\), because \(S_2\) is infinite and all
\(A_s\) act on separable spaces. Consequently, \((\alpha,\beta) \in
\Xi(A)\).\par
Finally, if \(\beta < \OPN{eig}(A)\) (and still \(\alpha = \dim(H) =
\OPN{cdr}(A)\), we can find a linear subspace \(E\) of \(H\) of (finite)
dimension \(\beta\) that is invariant for \(A\). It easily follows e.g. from
condition (c) of \THM{pure} that there is \(Z \in \zzZ(A)\) such that \(\dim(Z)
= \alpha\) and \([Z]_A = E^{\perp}\). This yields \((\alpha,\beta) \in \Xi(A)\)
and we are done.
\end{proof}

\section{Unbounded positive operators}

Recall that a (densely defined) operator \(T\dd \DdD(T) \to H\) (where \(\DdD(T)
\subset H\)) is said to be \textit{positive} if \(\scalar{Tx}{x}_H \geqsl 0\)
for any \(x \in \DdD(T)\). We emphasize that, according to the above definition,
positive operators need not be self-adjoint.\par
The main aim of this section is the following consequence of the results
presented in previous sections.

\begin{thm}{unbd}
For a (possibly unbounded) positive self-adjoint operator \(A\) in a Hilbert
space \(H\) \tfcae
\begin{enumerate}[\upshape(i)]
\item \(A\) is a unique positive closed densely defined operator \(T\) in \(H\)
 such that \(|T| = A\);
\item \(A\) is bounded.
\end{enumerate}
\end{thm}
\begin{proof}
Assume \(A\) is not bounded and let \(E\) denote the spectral measure of \(A\).
We will show that there is an isometry \(V\) on \(H\) such that
\(\dim(\RrR(V)^{\perp}) = 1\) and \(T \df VA\) is a positive closed operator
such that \(|T| = A \neq T\). To this end, set \(H_0 = \RrR(E([1,\infty)))\) and
\(H_1 = H_0^{\perp}\). \(A\) decomposes as \(A = A_0 \oplus A_1\) where \(A_j\)
is a positive self-adjoint operator in \(H_j\) (for \(j=1,2\)). Moreover,
\(A_0\) is not bounded, \(\scalar{A_0 x}{x}_H \geqsl \|x\|^2\) for \(x \in
\DdD(A_0)\) and \(\RrR(A_0) = H_0\). It is sufficient to show our claim (stated
at the beginning of the proof) for \(A_0\) (because if \(V_0\) is an appropriate
isometry for \(A_0\), then \(V_0 \oplus I_{H_1}\) is appropriate for \(A\)) and
thus we may and do assume that \(A = A_0\) (and \(H = H_0\)). Let \(B = A^{-1}
\in \Bb_+(H)\). It follows that \(\RrR(B) \neq H\) and we infer from
\PRO{iso-lin} that there is an isometry \(V \in \vvV(B)\) such that
\(\dim(\RrR(V)^{\perp}) = 1\). For any \(x \in \DdD(A)\) set \(y = Ax\) and
observe that \[\scalar{VAx}{x}_H = \scalar{Vy}{By}_H = \scalar{BVy}{y}_H \geqsl
0,\] which shows that \(T = VA\) is positive. Moreover, since \(A\) is closed
and \(V\) is isometric, \(T\) is closed as well and \(T^* = AV^*\). Thus \(T^*T
= A^2\) and consequently \(|T| = A\). Finally, since \(V \neq I\) and
\(\RrR(A)\) is dense in \(H\), we have \(T \neq A\). This shows that (i) is
followed by (ii). The reverse implication is well-known and left to the reader
as a simple exercise.
\end{proof}

As a consequence of the above theorem, we get the following, a little bit
surprising, result.

\begin{cor}{diag}
For a function \(u\dd X \to (0,\infty)\) defined on a non-empty set \(X\) \tfcae
\begin{enumerate}[\upshape(i)]
\item \(u\) is unbounded;
\item there exists a positive self-adjoint operator \(A\) in \(\ell_2(X)\) such
 that all the following conditions are fulfilled:
 \begin{enumerate}[\upshape(O1)]
 \item the domain of \(A\) contains the canonical basis \(\{e_x\dd\ x \in X\}\);
 \item \(\{\frac{A e_x}{u(x)}\dd\ x \in X\}\) is an orthonormal system in
  \(\ell_2(X)\);
 \item \(A(e_z) \neq u(z) e_z\) for some \(z \in X\).
 \end{enumerate}
\end{enumerate}
\end{cor}
\begin{proof}
First assume that \(u\) is bounded and \(A\) is a positive self-adjoint operator
such that (O1) and (O2) hold. We shall show that \(Ae_x = u(x) e_x\) for any \(x
\in X\) (that is: that \(A\) is bounded and diagonal in the canonical basis).
For simplicity, set \(f_x \df \frac{A e_x}{u(x)}\) for \(x \in X\). From (O2)
and the boundedness of \(u\) it easily follows that \(A\restriction{\lin\{e_x\dd
x \in X\}}\) is bounded. Consequently, \(A\) is bounded as well. Moreover, for
any \(x, y \in X\) we have \[\scalar{A^2 e_x}{e_y}_{\ell_2(X)} =
\scalar{Ae_x}{Ae_y}_{\ell_2(X)} = u(x) \overline{u(y)}
\scalar{f_x}{f_y}_{\ell_2(X)} = u(x)\overline{u(y)}
\scalar{e_x}{e_y}_{\ell_2(X)}\]
and hence \(A^2 e_x = u(x)^2 e_x\). Since \(A\) is a \textbf{unique} positive
square root of \(A^2\), we obtain that \(A e_x = u(x) e_x\) for any \(x \in
X\)---as claimed above.\par
Finally, assume \(u\) is unbouded and let \(B\) be the diagonal operator (with
respect to the canonical basis) induced by \(u\); that is, \(\DdD(B) = \{f \in
\ell_2(X)\dd\ uf \in \ell_2(X)\}\) and \(Bf = uf\) for \(f \in \DdD(B)\). Since
\(B\) is not bounded, we conclude from the proof of \THM{unbd} that there is
an isometry \(V\) such that \(\dim(\RrR(V)^{\perp}) = 1\) and \(T \df VB\) is
a positive closed operator. Hence the vectors \(\frac{T(e_x)}{u(x)} = V(e_x)\)
form an orthonormal system different from the canonical one. (Note also that
it is possible to enlarge this system by adding a single vector to obtain
an orthonormal basis of \(\ell_2(X)\)---since \(\dim(\RrR(V)^{\perp}) = 1\).)
Now to end the proof, it suffices to apply the Friedrichs' theorem \cite{fri} on
extending positive operators (consult also, e.g., \cite{a-s} or Theorem~5.38 in
\cite{wei}) to get a positive self-adjoint operator \(A\) in \(\ell_2(X)\) that
extends \(T\) and satisfies conditions (O1)--(O3).
\end{proof}

We leave it to the reader as an exercise that whenever conditions (O1)--(O3)
are fulfilled (for a positive operator \(A\)), the system in (O2) is never
an orthonormal basis. However, as the above proof shows, if only \(u\) is
an unbounded function, we can always find such an operator \(A\) for which
the closed linear subspace generated by the system from (O2) has codimension
\(1\).

We end this section with the following result which is unrelated with the main
subject of the paper (that is, it says nothnig about positivity). We will use it
in the next section. It is likely that this result is already known. However, we
could not find it in the literature. It can be considered interesting in its own
right.

\begin{lem}{closed}
Let \(T\) be a closable operator and \(S \df \{x \in \DdD(T)\dd\ \|Tx\| \leqsl
1\}\).
\begin{enumerate}[\upshape(a)]
\item \(T\restriction{S}\) is continuous in the weak topologies, and \(S\) is
 closed in \(\DdD(T)\).
\item If \(T\) is closed, then \(S\) is a closed set.
\end{enumerate}
\end{lem}

\begin{rem}{reflex}
\LEM{closed} (and its proof without any change) is valid for all (not
necessarily densely defined) closable linear operators \(T\dd \DdD(T) \to Y\)
where \(\DdD(T)\) is a subspace of a Banach space \(X\) and \(Y\) is
a \textbf{reflexive} Banach space.
\end{rem}

\begin{proof}[Proof of \LEM{closed}]
Replacing \(T\) by \(\bar{T}\), we may and do assume \(T\) is closed. In that
case we need to show that \(T\restriction{S}\) is continuous in the weak
topologies, and \(S\) is closed. To this end, let
\(\pmb{x} = (x_{\sigma})_{\sigma\in\Sigma}\) be a net in \(S\) that weakly
converges to some \(z \in H\) (where \(H\) is the underlying space containing
the domain of \(T\)). Let \((y_{\lambda})_{\lambda\in\Lambda}\) be any subnet of
\(\pmb{x}\) such that \((T y_{\lambda})_{\lambda\in\Lambda}\) is weakly
convergent, say to \(w\). Then \((y_{\lambda},T y_{\lambda})_{\lambda\in\Lambda}
\subset \Gamma(T)\) is a net weakly convergent (in the product space) to
\((z,w)\). Since norm closed convex subsets of Banach spaces are weakly closed,
we conclude that \((z,w) \in \Gamma(T)\); that is, \(z \in \DdD(T)\) and \(Tz =
w\). This shows that each weakly convergent subnet of
\((T x_{\sigma})_{\sigma\in\Sigma}\) converges to \(Tz\). So, it follows from
the weak compactness of the unit ball (in the target space) that the net
\((T x_{\sigma})_{\sigma\in\Sigma}\) itself converges to \(Tz\). In particular,
\(\|Tz\| \leqsl 1\) and we are done.
\end{proof}

\section{Positive definite kernels}

Before passing to the main issue of this part, first we recall necessary
notions.\par
A \textit{kernel} on \(X\) (where \(X\) is an arbitrary non-empty set) is
a complex-valued function on \(X \times X\). A kernel \(K\dd X \times X \to
\CCC\) is said to be a \textit{positive definite kernel} (on \(X\)) (or
a \textit{Hilbert space reproducing kernel}, or briefly a \textit{reproducing
kernel}) if
\begin{equation}\label{eqn:pd}
\sum_{j,k=1}^n \lambda_j \bar{\lambda}_k K(x_j,x_k) \geqsl 0
\end{equation}
for any \(n \geqsl 1\), \(x_1,\ldots,x_n \in X\) and
\(\lambda_1,\ldots,\lambda_n \in \CCC\). Note that the above condition says that
the sum on the left-hand side of \eqref{eqn:pd} (whose summands are complex!) is
a non-negative \textbf{real} number. It is well-known (and easy to check) that
for any reproducing kernel \(K\) on \(X\),
\begin{equation}\label{eqn:anti}
K(y,x) = \overline{K(x,y)}
\end{equation}
and \(K(x,x) \geqsl 0\) for all \(x, y \in X\). It is well-known (and can
briefly be proved by applying Sylvester's theorem on strictly positive definite
matrices) that a kernel \(K\dd X \times X \to \CCC\) is positive definite iff
\(K\) satisfies \eqref{eqn:anti} and \(\det[K(x_j,x_k)]_{j,k=1}^n \geqsl 0\) for
all \(x_1,\ldots,x_n \in X\) (and arbitrary \(n \geqsl 1\)). Another equivalent
(and well-known) condition for the kernel \(K\) to be a reproducing kernel is
the existence of a function \(j\dd X \to H\) where \(H\) is some Hilbert space
such that \(K(x,y) = \scalar{j(x)}{j(y)}_H\). To shorten statements, below we
will use the abbreviation ``pd'' for ``positive definite.'' We will also write
``\(K \gg 0\)'' to express that \(K\) is a pd kernel. More generally, for two
kernels \(K\) and \(L\) defined on a common set the notations ``\(K \ll L\)''
and ``\(L \gg K\)'' will mean that \(L-K\) is a pd kernel.\par
In this paper we study those pd kernels that have a pdms root (see \DEF{root}).
Note that, thanks to \eqref{eqn:anti}, \eqref{eqn:matrix} is equivalent to
\(L(x,z) = \sum_{y \in X} K(x,y) \overline{K(z,y)}\). In particular, if \(z =
x\), the series in \eqref{eqn:matrix} has non-negative summands and thus its sum
is well-defined (it is either a real number or \(\infty\)). And, if all these
series (with \(z = x\)) are finite, the Schwarz inequality shows that
the right-hand side series in \eqref{eqn:matrix} converges for arbitrary \(x, z
\in X\). So, there is no ambiguity in understanding this formula. To simplify
further statements, let us introduce the following

\begin{dfn}{ell2}
A kernel \(K\) on \(X\) is called an \textit{\(\ell_2\)-kernel} if both
\(K(x,\cdot)\) and \(K(\cdot,x)\) are members of \(\ell_2(X)\) for any \(x \in
X\).\par
For two \(\ell_2\)-kernels \(K\) and \(L\) on \(X\), \(K * L\) is a kernel on
\(X\) given by
\[(K * L)(x,z) = \sum_{y \in Y} K(x,y) L(y,z) \qquad (x, z \in X).\]
In addition, \(K^*\) is a kernel on \(X\) such that \(K^*(x,y) =
\overline{K(y,x)}\) for any \(x, y \in X\).
\end{dfn}

The paragraph preceding the above definition explains that for any
\(\ell_2\)-kernels \(K\) and \(L\) (both on \(X\)), \(K * L\) is a well-defined
kernel on \(X\). Moreover, we readily have \((K * L)(x,y) =
\scalar{K(x,\cdot)}{L^*(y,\cdot)}_{\ell_2(X)}\). It is worth noting here that,
in general, \(K * L\) is not an \(\ell_2\)-kernel (see, e.g. item (A) in
\EXM{c0} below). Observe also that a pd kernel \(K\) on \(X\) is
an \(\ell_2\)-kernel iff \(K(x,\cdot) \in \ell_2(X)\) for all \(x \in X\).\par
Using the notation introduced in \DEF{ell2}, we can reformulate the equation
defining a pdms root as follows: \(K \gg 0 \) is a pdms root of \(L \gg 0\) iff
\(K\) is an \(\ell_2\)-kernel and \(L = K * K\). In the last statement
the assumption that \(L \gg 0\) can be skipped as shown by the following very
easy

\begin{pro}{square}
For any \(\ell_2\)-kernel \(K\), \(K * K^*\) is a pd kernel. In particular,
if \(K \gg 0\) is an \(\ell_2\)-kernel, then \(K * K \gg 0\) as well.
\end{pro}
\begin{proof}
Each pd kernel \(K\) satisfies \(K = K^*\), so it is sufficient to prove
the first claim which follows from the formula \((K * K^*)(x,y) =
\scalar{K(x,\cdot)}{K(y,\cdot)}_{\ell_2(X)}\).
\end{proof}

For any kernel \(K\) on \(X\) let \(K_{\op}\) denote a unique linear operator
defined on \(\ell_{fin}(X)\) such that \(K_{\op}(e_x) = K(x,\cdot)\) for all
\(x \in X\) (so, all values of \(K_{\op}\) are complex-valued functions on
\(X\)). Everywhere below its domain \(\DdD(K_{\op}) = \ell_{fin}(X)\) will
always be equipped with the norm and the topology inherited from \(\ell_2(X)\)
and considered as a subspace of \(\ell_2(X)\). In contrast, the target space of
\(K_{\op}\) will vary and will always be specified.

\begin{lem}{closable}
For any \(\ell_2\)-kernel \(K\) on \(X\), \(K_{\op}\dd \ell_{fin}(X) \to
\ell_2(X)\) is a closable operator such that \(K(x,y) =
\scalar{K_{\op}e_x}{e_y}_{\ell_2(X)}\) for any \(x, y \in X\).
\end{lem}
\begin{proof}
The only thing that needs proving is the closability of \(K_{\op}\). But this
easily follows from the relation: \(\scalar{K_{\op} f}{e_x}_{\ell_2(X)} =
\scalar{f}{K^*(x,\cdot)}_{\ell_2(X)}\ (x \in X,\ f \in \ell_{fin}(X))\).
The details are left to the reader.
\end{proof}

Every pd kernel \(K\) on \(X\) generates a unique Hilbert function space
(consisting of complex-valued functions on \(X\)), to be denoted by \(\hhH_K\),
such that the following two conditions are fulfilled:
\begin{itemize}
\item the functions \(K(x,\cdot)\) (where \(x\) runs over all elements of \(X\))
 belong to \(\hhH_K\);
\item \(\scalar{f}{K(x,\cdot)}_{\hhH_K} = f(x)\) for any \(x \in X\) and \(f \in
 \hhH_K\).
\end{itemize}
In particular, \(K(x,y) = \scalar{K(x,\cdot)}{K(y,\cdot)}_{\hhH_K}\). (Very
often \(\hhH_K\) is defined by conditions obtained from the above two by
replacing the functions \(K(x,\cdot)\) by \(K(\cdot,x)\). In general, this way
leads to a \underline{different} vector space. However, both the approaches are
fully equivalent and it is a matter of taste which one to choose. Our choice is
more convenient for our purposes.) A well-known fact says that \(\hhH_{\delta_X}
= \ell_2(X)\), which we will use many times without any additional
explanations. (Recall that \(\delta_X\) is the pd kernel on \(X\) such that
\(\delta_X(x,x) = 1\) for all \(x \in X\) and \(\delta_X(x,y) = 0\)
whenever \(x \neq y\).)\par
To simplify further statements, let us say that an operator \(T\dd \DdD(T) \to
H\) (where \(\ell_{fin}(X) \subset \DdD(T) \subset \ell_2(X)\) and \(H\) is
a Hilbert space) \textit{factorizes} a pd kernel \(K\) on \(X\) if
\begin{equation}\label{eqn:factor}
K(x,y) = \scalar{Te_x}{Te_y}_H \qquad (x, y \in H).
\end{equation}

\begin{lem}{forall}
For a pd kernel \(K\) on \(X\) \tfcae
\begin{enumerate}[\upshape(i)]
\item there exists a closable operator that factorizes \(K\);
\item the restriction to \(\ell_{fin}(X)\) of each operator that factorizes
 \(K\) is closable.
\end{enumerate}
\end{lem}
\begin{proof}
It is easy to see that (i) is implied by (ii) (note that it is sufficient to
show the existence of an operator that factorizes \(K\)): the operator
\(K_{\op}\dd \ell_{fin}(X) \to \hhH_K\) factorizes \(K\). The reverse
implication is also simple: if \(S\dd \DdD(S) \to H\) and \(T\dd \DdD(T) \to E\)
factorize \(K\), then \(\scalar{Te_x}{Te_y}_E = \scalar{Se_x}{Se_y}_H\) for any
\(x, y \in X\) and therefore there exists a linear isometry \(V\dd
T(\ell_{fin}(X)) \to H\) such that \(Sf = VTf\) for any \(f \in \ell_{fin}(X)\).
Consequently, \(S\restriction{\ell_{fin}(X)}\) is closable iff
\(T\restriction{\ell_{fin}(X)}\) is so, and we are done.
\end{proof}

For a collection of kernels \(\{K_s\dd X_s \times X_s \to \CCC\}_{s \in S}\)
where the sets \(X_s\) are all pairwise disjoint we define the kernel
\(\bigoplus_{s \in S} K_s\) on the (disjoint) union \(\bigsqcup_{s \in S} X_s\)
of \(X_s\) as follows:
\begin{itemize}
\item \((\bigoplus_{s \in S} K_s)(x,y) = K_t(x,y)\) for \(x, y \in X_t\) and
 arbitrary \(t \in S\);
\item \((\bigoplus_{s \in S} K_s)(x,y) = 0\) if \(x \in X_p\), \(y \in X_q\) and
 \(p \neq q\).
\end{itemize}
It is easy to check and left to the reader that \(\bigoplus_{s \in S} K_s \gg
0\) iff \(K_s \gg 0\) for all \(s \in S\).\par
For simplicity, let us call a vector \(u \in \ell_{fin}(X)\) \textit{\(K\)-unit}
(where \(K\) is a pd kernel on \(X\)) if \(\sum_{x,y \in X} u(x)\overline{u(y)}
K(x,y) \leqsl 1\).\par
Now we gather various criteria for a pd kernel to have a pdms root.

\begin{thm}{exists}
For a pd kernel \(K\) on \(X\) \tfcae
\begin{itemize}
\item[(a1)] \(K\) has a pdms root;
\item[(a2)] there exists an \(\ell_2\)-kernel \(L\) on \(X\) such that \(K = L *
 L^*\);
\item[(b)] the set \(X\) admits a decomposition \(X = \bigsqcup_{s \in S} X_s\)
 into (pairwise disjoint) non-empty (at most) \textbf{countable} sets such that
 \(K = \bigoplus_{s \in S} K\restriction{X_s \times X_s}\) and each of
 \(K\restriction{X_s \times X_s}\) has a pdms root;
\item[(c1)] there exists a closable operator that factorizes \(K\);
\item[(c2)] there exists a positive self-adjoint operator \(B\) in \(\ell_2(X)\)
 that factorizes \(K\) and has \(\ell_{fin}(X)\) as a core;
\item[(c3)] \(K_{\op}\dd \ell_{fin}(X) \to \hhH_K\) is closable;
\item[(d)] \(\ell_2(X) \cap \hhH_K\) is dense in \(\hhH_K\);
\item[(e1)] whenever a sequence \((\alpha_n)_{n=1}^{\infty} \subset
 \ell_{fin}(X)\) norm converges to \(0\) and consists of \(K\)-unit vectors,
 then \(\lim_{n\to\infty} \sum_{x \in X} \alpha_n(x) K(x,z) = 0\) for all \(z
 \in X\);
\item[(e2)] for any \(z \in X\) and \(\epsi > 0\) there is a finite non-empty
 set \(F \subset X\) such that the following condition holds: if a \(K\)-unit
 vector \(f \in \ell_{fin}(X)\) vanishes at each point of \(F\) and has norm not
 exceedind \(1\), then \(|\sum_{x \in X} f(x) K(x,z)| \leqsl \epsi\);
\item[(e3)] for any \(z \in X\) and \(\epsi > 0\) there are a finite orthonormal
 system \(u_1,\ldots,u_k\) in \(\ell_2(X)\) and \(\delta > 0\) such that
 the following condition is fulfilled: if a \(K\)-unit vector \(v \in
 \ell_{fin}(X)\) satisfies \(|\scalar{v}{u_j}_{\ell_2(X)}| \leqsl \delta\) for
 \(j=1,\ldots,k\), then \(|\sum_{x \in X} v(x) K(x,z)| \leqsl \epsi\).
\end{itemize}
\end{thm}
\begin{proof}
For the reader's convenience, let us draw a scheme of the proof:
(a1)\(\implies\)(c1)\(\implies\)(c2)\(\implies\)(a1);
(a1)\(\implies\)(a2)\(\implies\)(c1); (c1)\(\iff\)(c3)\(\iff\)(d);
(c3)\(\implies\)(e3)\(\implies\)(e2)\(\implies\)(e1)\(\implies\)(c3) and finally
(a1)\(\implies\)(b)\(\implies\)(a1).\par
First assume \(L \gg 0\) satisfies \(K = L * L\) (see (a1)). This means that
\[L_{\op}\dd \ell_{fin}(X) \to \ell_2(X)\] factorizes \(K\). It follows from
\LEM{closable} that \(L_{\op}\) is closable (recall that \(L\) is an
\(\ell_2\)-kernel). So, (c1) holds. Let us now check that (c2) is implied by
(c1). To this end, let \(T\dd \ell_{fin}(X) \to H\) be closable and factorize
\(K\), and let \(\bar{T} = Q B\) be the polar decomposition of \(\bar{T}\). Then
\(B\) is a positive self-adjoint operator whose domain contains
\(\ell_{fin}(X)\) and \(Q\) is isometric on \(\overline{\RrR}(B)\). So, \(K(x,y)
= \scalar{Te_x}{Te_y}_H = \scalar{QBe_x}{QBe_y}_H =
\scalar{Be_x}{Be_y}_{\ell_2(X)}\). Moreover, since \(\ell_{fin}(X)\) is a core
of \(\bar{T}\), it is a core of \(B\) as well---that is, (c2) holds. Now we will
show that (a1) follows from (c2). So, let \(B\) be as specified in (c2) and
define \(L\dd X \times X \to \CCC\) by \(L(x,y) =
\scalar{Be_x}{e_y}_{\ell_2(X)}\ (x, y \in X)\). Since \(B\) is positive, it is
readily seen that \(L \gg 0\). Moreover, observe that \(B e_x = \sum_{y \in X}
L(x,y) e_y\) and therefore \(L\) is an \(\ell_2\)-kernel and
\(\scalar{L(x,\cdot)}{L(y,\cdot)}_{\ell_2(X)} = \scalar{Be_x}{Be_y}_{\ell_2(X)}
= K(x,y)\) and (a1) is fulfilled.\par
Of course, (a1) is followed by (a2). Conversely, if \(L\) is as specified in
(a2), then \(L_{\op}\dd \ell_{fin}(X) \to \ell_2(X)\) is closable (by
\LEM{closable}) and factorizes \(K\), which shows (c1).\par
Further, if (c1) holds, then \(K_{\op}\dd \ell_{fin}(X) \to \hhH_K\) is closable
as it factorizes \(K\) (see \LEM{forall} and its proof), which shows that (c1)
implies (c3). The reverse implication is trivial. To show that (c3) is
equivalent to (d), consider \(T = K_{\op}\dd \ell_{fin}(X) \to \hhH_K\) and
recall that (since \(T\) is densely defined) \(T\) is closable iff \(T^*\) is
densely defined. Therefore it is sufficient to check that \(\DdD(T^*) =
\ell_2(X) \cap \hhH_K\). To this end, observe that \(g \in \hhH_K\) belongs to
\(\DdD(T^*)\) and \(T^*g = f \in \ell_2(X)\) iff \(\scalar{Te_x}{g}_{\hhH_K} =
\overline{f(x)}\) for all \(x \in X\). Equivalently, we need to have \(f(x) =
\scalar{g}{K(x,\cdot)}_{\hhH_K} = g(x)\). So, \(g \in \DdD(T^*)\) iff \(g \in
\ell_2(X) \cap \hhH_K\) (and then \(T^*g = g\)), which finishes this part of
the proof.\par
Going further, for simplicity, we denote by \(S \subset \ell_{fin}(X)\) the set
of all \(K\)-unit vectors, and by \(T\) the operator \(K_{\op}\dd \ell_2(X) \to
\hhH_K\). Observe that for any \(u \in \ell_{fin}(X)\),
\[u \in S \iff \|Tu\| \leqsl 1\]
and \(\sum_{x \in X} u(x) K(x,y) = \scalar{Tu}{K(y,\cdot)}_{\hhH_K}\) (for any
\(y \in X)\). Now if (c3) holds, it follows from \LEM{closed} that (e3) is
fulfilled. Indeed, a note that all sets of the form
\begin{equation}\label{eqn:aux22}
\{v \in \ell_{fin}(X)\dd\ |\scalar{v}{u_j}| < \delta\ (j=1,\ldots,k)\}
\end{equation}
(where \(\delta > 0\) and \(u_1,\ldots,u_k\) is a finite orthonormal system in
\(\ell_2(X)\)) form a neighbourhood basis of \(0\) in the weak topology of
\(\ell_2(X)\) enables one deducing (e3) from (c3) (we leave the details to
the reader). Further, if (e3) holds, then (e2) holds as well (for in (e2)
\(K\)-unit vectors are taken from the unit ball of \(\ell_{fin}(X)\) and on
the unit ball of \(\ell_2(X)\) a neighbourhood basis of \(0\) in the weak
topology can be formed by the sets \eqref{eqn:aux22} where all vectors
\(u_1,\ldots,u_k\) are taken from the canonical basis of \(\ell_2(X)\)). Let us
now give a more detailed proof that (e1) follows from (e2). To this end, assume
(e2) holds and let a sequence \((\alpha_n)_{n=1}^{\infty}\) be as specified in
(e1). Fix \(z \in X\) and \(\epsi > 0\). Choose a finite set \(F \subset X\)
guaranteed by (e2) (for these \(z\) and \(\epsi > 0\)). Denote by \(E\)
the linear span of all \(e_x\) with \(x \in F\) and write \(\alpha_n = \beta_n
+ \gamma_n\) where \(\beta_n \in E\) and \(\gamma_n \perp E\). Then
\(\lim_{n\to\infty} \beta_n = \lim_{n\to\infty} \gamma_n = 0\). Since \(E\) is
finite-dimensional,
\begin{equation}\label{eqn:aux23}
\lim_{n\to\infty} T\beta_n = 0.
\end{equation}
Therefore \(\beta_n \in S\) for sufficiently large \(n\). Consequently,
\(\frac12 \gamma_n = \frac12(\alpha_n-\beta_n)\) eventually belongs to \(S\) as
well. Since \(\gamma_n\) vanishes at each point of \(F\), we infer from (e2)
that \(|\scalar{T\gamma_n}{K(z,\cdot)}_{\hhH_K}| \leqsl 2\epsi\) for
sufficiently large \(n\). This inequality combined with \eqref{eqn:aux23} yields
\(|\scalar{T\alpha_n}{K(z,\cdot)}_{\hhH_K}| \leqsl 3\epsi\) (for sufficiently
large \(n\)) and hence (e1) is fulfilled. Finally, assume (e1) holds. Our aim is
to show that (c3) is fulfilled; that is, that \(T\) (defined above) is closable.
To this end, assume \((\alpha_n)_{n=1}^{\infty} \subset \ell_{fin}(X)\) norm
converges to \(0\) and \(T\alpha_n \to \beta \in \hhH_K\ (n \to \infty)\). We
need to check that \(\beta = 0\). Without loss of generality, we may and do
assume that \(\|T \alpha_n\| \leqsl 1\). So, \(\alpha_n \in S\) and it follows
from (e1) that \(\lim_{n\to\infty} \scalar{T\alpha_n}{K(z,\cdot)}_{\hhH_K} = 0\)
for any \(z \in X\), from which it easily follows that \(\beta = 0\).\par
It remains to check that (b) is equivalent to (a1). First assume (a1) holds and
choose any pd \(\ell_2\)-kernel \(L\) such that \(K = L * L\). Define
an equivalence relation ``\(\sim\)'' on \(X\) as follows: \(x \sim y\) if there
are points \(a_0,\ldots,a_N \in X\) (for some \(N > 0\)) such that \(a_0 = x\),
\(a_N = y\) and \(L(x_{j-1},x_j) \neq 0\) for \(j=1,\ldots,N\). Observe that
all equivalence classes \([x]_{\sim}\) are at most countable (because the set
\(\{x \in X\dd\ L(x,y) \neq 0\}\) is such for any \(y \in X\)). So, we can
divide \(X\) into pairwise disjoint sets \(X_s\) such that \(L(x,y) = 0\) for
any \(x \in X_s\) and \(y \in X_{s'}\) with distinct \(s, s' \in S\) (namely,
\(\{X_s\dd\ s \in S\} = \{[x]_{\sim}\dd\ x \in X\}\)). For simplicity, set \(L_s
\df L\restriction{X_s \times X_s}\) and note that \(L = \bigoplus_{s \in S}
L_s\) and each of \(L_s\) is a pd \(\ell_2\)-kernel. It is then easy to verify
that \(L * L = \bigoplus_{s \in S} (L_s * L_s)\) and therefore (b) holds.
Conversely, if (b) is satisfied, then for each \(s \in S\) we can choose a pd
\(\ell_2\)-kernel \(L_s\) on \(X_s\) such that \(K\restriction{X_s \times X_s} =
L_s * L_s\). Then it suffices to set \(L \df \bigoplus_{s \in S} L_s\) to get
a pd \(\ell_2\)-kernel such that \(K = L *L\).
\end{proof}

In \PRO{roots} below we will show that the operator \(B\) witnessing the above
condition (c2) is unique.\par
In the next theorem we will make use of the following two results. The former
of them is well-known (see, e.g., Theorem~6 on page 37 in \cite{sai}) and it is
likely that so is the latter, but we could not find it in the literature and
thus we give its proof.

\begin{lem}{classic}
For two pd kernels \(K\) and \(L\) on a common set and a constant \(c \geqsl 0\)
\tfcae
\begin{enumerate}[\upshape(i)]
\item \(\hhH_K \subset \hhH_L\) and the identity operator from \(\hhH_K\) into
 \(\hhH_L\) has norm not exceeding \(c\);
\item \(K \ll c^2 L\).
\end{enumerate}
Moreover, \(\hhH_K \subset \hhH_L\) iff \textup{(ii)} hols for some \(c > 0\).
\end{lem}

\begin{lem}{incr}
Let \(\{K_{\sigma}\}_{\sigma\in\Sigma}\) be an increasing net of pd kernels on
\(X\), that is:
\begin{itemize}
\item \((\Sigma,\leqsl)\) is a directed set;
\item \(K_{\sigma} \ll K_{\tau}\) for any \(\sigma, \tau \in \Sigma\) with
 \(\sigma \leqsl \tau\).
\end{itemize}
If \(K\dd X \times X \to \CCC\) is the pointwise limit of this net---that is, if
\[K(x,y) = \lim_{\sigma\in\Sigma} K_{\sigma}(x,y) \qquad (x, y \in X),\] then
\(\bigcup_{\sigma\in\Sigma} \hhH_{K_{\sigma}}\) is a dense linear subspace of
\(\hhH_K\).
\end{lem}
\begin{proof}
For simplicity and to avoid confusion, for any \(\sigma \in \Sigma\) we denote
by \(\scalarr_{\sigma}\) and \(\|\cdot\|_{\sigma}\) the scalar product and,
respectively, the norm of \(\hhH_{K_{\sigma}}\); whereas \(\scalarr\) and
\(\|\cdot\|\) stands for the scalar product and the norm of \(\hhH_K\).
Additionally, we set \(H \df \bigcup_{\sigma\in\Sigma} \hhH_{K_{\sigma}}\).\par
Since \(K_{\sigma} \ll K_{\tau}\) whenever \(\sigma \leqsl \tau\), after passing
with \(\tau\) to the limit, we easily get \(K_{\sigma} \ll K\). It follows from
\LEM{classic} that for any \(\sigma, \tau \in \Sigma\):
\begin{itemize}
\item \(\hhH_{K_{\sigma}} \subset \hhH_{K_{\tau}}\) whenever \(\sigma \leqsl
 \tau\); consequently, since \(\Sigma\) is directed, \(H\) is a linear space;
\item \(\hhH_{K_{\sigma}} \subset \hhH_K\) and \(\|I_{\sigma}\| \leqsl 1\) where
 \(I_{\sigma}\dd \hhH_{K_{\sigma}} \to \hhH_K\) is the identity operator.
\end{itemize}
Fix \(x \in X\). Then \[\|K_{\sigma}(x,\cdot)\|^2 \leqsl
\|K_{\sigma}(x,\cdot)\|_{\sigma}^2 =
\scalar{K_{\sigma}(x,\cdot)}{K_{\sigma}(x,\cdot)}_{\sigma} = K_{\sigma}(x,x)
\leqsl K(x,x).\]
So, the net \(\{K_{\sigma}(x,\cdot)\}_{\sigma\in\Sigma} \subset H\) is bounded
in \(\hhH_K\). Moreover, for any \(z \in X\) we have \[\lim_{\sigma\in\Sigma}
\scalar{K_{\sigma}(x,\cdot)}{K(z,\cdot)} = \lim_{\sigma\in\Sigma}K_{\sigma}(x,z)
= K(x,z) = \scalar{K(x,\cdot)}{K(z,\cdot)}.\] Since the functions \(K(z,\cdot)\
(z \in X)\) form a total subset of \(\hhH_K\), it follows from the above
convergence and the boundedness of the net under consideration that
\(K_{\sigma}(x,\cdot)\) weakly converge to \(K(x,\cdot)\). Therefore
\(K(x,\cdot)\) belongs to the weak closure of \(H\) which coincides with
the norm closure of \(H\). Consequently, \(\hhH_K =
\overline{\lin}\{K(x,\cdot)\dd\ x \in X\}\) is contained in the (norm) closure
of \(H\) and we are done.
\end{proof}

As a consequence of \THM{exists}, we obtain the following

\begin{thm}{pdms}
Let \(K\) be a pd kernel on \(X\).
\begin{enumerate}[\upshape(I)]
\item If \(K\) has a pdms root and \(A\) is a non-empty subset of \(X\), then
 \(K\restriction{A \times A}\) has a pdms root as well.
\item If \(K = \bigoplus_{s \in S} K_s\), then \(K\) has a pdms root iff each of
 \(K_s\) has a pdms root.
\item If \(K\) has a pdms root and \(u\dd X \to \CCC\) is a bounded function,
 then the kernel \(L\dd X \times X \ni (x,y) \mapsto \overline{u(x)} u(y) K(x,y)
 \in \CCC\) has a pdms root as well. In particular, the kernel
 \begin{equation}\label{eqn:bdd}
 K_{\OPN{bd}}\dd X \times X \ni (x,y) \mapsto
 \frac{K(x,y)}{\sqrt{\max(K(x,x),1) \cdot \max(K(y,y),1)}} \in \CCC
 \end{equation}
 has a pdms root provided so has \(K\).
\item If \(K\) is bounded and has a pdms root, then for any \(x \in X\),
 \(K(x,\cdot)\) is a \(c_0\)-function; that is, for any \(x \in X\) and \(\epsi
 > 0\) there is a finite set \(F \subset X\) such that \(|K(x,y)| < \epsi\) for
 any \(y \notin F\).
\item If \(K\) is an \(\ell_2\)-kernel, it has a pdms root. In particular, each
 pdms root of a pd kernel has a pdms root.
\item If \(K = \sum_{s \in S} K_s\) where \(\{K_s\}_{s \in S}\) is an arbitrary
 family of pd kernels having pdms roots, then \(K\) also has a pdms root.
\item If \(K\) is the pointwise limit of an increasing net of pd kernels (cf.\
 the statement of \LEM{incr}) each of which has a pdms root, then also \(K\) has
 a pdms root.
\item If \(L\) is a pd kernel on \(X\) such that \(c_1 K \ll L \ll c_2 K\) for
 some positive constants \(c_1\) and \(c_2\), then either both \(K\) and \(L\)
 have pdms roots or none of them.
\item There exists a unique pd kernel \(K_0 \ll K\) that has a pdms root and is
 the greatest kernel with these properties; that is, if \(L \ll K\) is a pd
 kernel that has a pdms root, then \(L \ll K_0\).
\item There exists a pd kernel \(K_1 \ll K\) that has a \textbf{unique} pdms
 root and the following property: whenever \(L \gg 0\) is such that \(L \ll a
 K\) for some constant \(a > 0\), then \(L\) has a unique pdms root iff \(L \ll
 b K_1\) for some constant \(b > 0\).
\end{enumerate}
\end{thm}

Before giving a proof, we explain how to understand the sum appearing in item
(VI) above and when this (generalized) series converges. The formula \(K =
\sum_{s \in S} K_s\) is understood pointwise: we only assume that \(K(x,y) =
\sum_{s \in S} K_s(x,y)\) for any \(x, y \in X\). In particular,
\begin{equation}\label{eqn:aux26}
\sum_{s \in S} K_s(x,x) < \infty \qquad (x \in X)
\end{equation}
(all summands in \eqref{eqn:aux26} are non-negative and hence the series is
well-defined as a quantity in \([0,\infty]\)). Conversely, if \eqref{eqn:aux26}
is fulfilled, then \(\sum_{s \in S} K_s(x,y)\) is absolutely convergent (for
any \(x, y \in X\)):
\[\sum_{s \in S} |K_s(x,y)| \leqsl \sum_{s \in S} \sqrt{K_s(x,x) K_s(y,y)}
\leqsl \sum_{s \in S} (K_s(x,x) + K_s(y,y)) < \infty\]
where the first inequality above follows from the property that the matrix
\[\begin{pmatrix}K(x,x) & K(x,y)\\K(y,x) & K(y,y)\end{pmatrix}\] is positive.
\begin{proof}[Proof of \THM{pdms}]
Since the restriction of a closable operator is closable as well, item (I)
immediately follows from condition (c1) in \THM{exists}, whereas (II) is
a consequence of (I) and of item (b) therein.\par
To prove that \(L\) defined in (III) has a pdms root, we use condition (c1) of
\THM{exists}. Since \(K\) has a pdms kernel, there is a closable operator \(T\dd
\ell_{fin}(X) \to H\) that factorizes \(K\). Let \(S\dd \ell_{fin}(X) \to
\ell_{fin}(X)\) be given by \(Se_x \df u(x) e_x\). Then \(S\) is bounded (since
so is \(u\)) and hence the operator \(TS\dd \ell_{fin}(X) \to H\) is closable.
Observe that \(TS\) factorizes \(L\) and thus \(L\) has a pdms root (by (c1)).
The claim about \(K_{\OPN{bd}}\) follows from the boundedness of the function
\(X \ni x \mapsto \max(K(x,x),1)^{-1/2} \in (0,\infty)\).\par
We turn to (IV). Assume \(K\) is bounded and has a pdms root. It is sufficient
to show that for any \(x \in X\), \(\lim_{n\to\infty} K(x,y_n) = 0\) for any
one-to-one sequence \((y_n)_{n=1}^{\infty} \subset X\). To this end, take
an upper bound \(M \geq 1\) of \(|K|\) and apply condition (e2) (in
\THM{exists}) to \(z = x\) and \(f = \frac1M e_{y_n}\): (e2) implies that
\(\lim_{n\to\infty} \frac1M K(y_n,x) = 0\). Consequently, \(K(x,y_n) =
\overline{K(y_n,x)}\) converges to \(0\) as \(n\) tends to \(\infty\).\par
Now assume \(K\) is an \(\ell_2\)-kernel. Then \(\lin\{K(x,\cdot)\dd\ x \in X\}
\subset \ell_2(X) \cap \hhH_K\) and hence \(\ell_2(X) \cap \hhH_K\) is dense in
\(\hhH_K\). So, condition (d) of \THM{exists} shows that \(K\) has a pdms root.
Since each pdms root is an \(\ell_2\)-kernel, the whole conclusion of (V)
follows.\par
To  prove (VI) we apply condition (c1) of \THM{exists}. So, for any \(s \in S\)
there exists a closable operator \(T_s\dd \ell_{fin}(X) \to H_s\) that
factorizes \(K_s\). Then, for any \(x \in X\), \(\sum_{s \in S} \|T_s e_x\|^2 =
\sum_{s \in S} K_s(x,x) = K(x,x) < \infty\) and therefore \(\oplus_{s \in S} T_s
e_x \in \bigoplus_{s \in S} H_s\). In particular, for any \(f \in
\ell_{fin}(X)\), \(Tf \df \oplus_{s \in S} T_s f \in \bigoplus_{s \in S} H_s\).
In this way we have defined a linear operator \(T\dd \ell_{fin}(X) \to H \df
\bigoplus_{s \in S} H_s\). It is readily seen that \(T\) is closable---as all
\(T_s\) are such. Moreover, for any \(x, y \in X\) we have
\[\scalar{Te_x}{Te_y}_H = \sum_{s \in S} \scalar{T_s e_x}{T_s e_y}_{H_s} =
\sum_{s \in S} K_s(x,y) = K(x,y)\]
which shows that \(T\) factorizes \(K\). So, an application of (c1) from
\THM{exists} completes the proof of (VI).\par
We turn to (VII). Let \(\{K_{\sigma}\}_{\sigma\in\Sigma}\) be an increasing net
(of pd kernels with pdms roots) whose pointwise limit is \(K\). For simplicity,
we set \(H_{\sigma} \df \hhH_{K_{\sigma}}\ (\sigma \in \Sigma)\) and \(H \df
\hhH_K\). It follows from \LEM{incr} that \(H_* \df \bigcup_{\sigma\in\Sigma}
H_{\sigma}\) is a dense subspace of \(H\). Moreover, \LEM{classic} yields that
the identity operator \(I_{\sigma}\dd H_{\sigma} \to H\) is continuous; whereas
condition (d) of \THM{exists} implies that \(\ell_2(X) \cap H_{\sigma}\) is
dense in \(H_{\sigma}\). Consequently, \(I_{\sigma}(\ell_2(X) \cap H_{\sigma})\)
is dense in \(I_{\sigma}(H_{\sigma})\) and therefore the closure of \(\ell_2(X)
\cap H\) (in \(H\)) contains \(H_*\), which finishes the proof of (VII).\par
Property (VIII) immediately follows from \LEM{classic} (which implies that under
the assumption of (VIII), \(\hhH_K = \hhH_L\) and their topologies coincide) and
condition (d) of \THM{exists} (since in that case \(\ell_2(X) \cap \hhH_K =
\ell_2(X) \cap \hhH_L\)).\par
To prove (IX), denote by \(P\) the orthogonal projection from \(\hhH_K\) onto
the closure \(E\) (in \(\hhH_K\)) of \(\ell_2(X) \cap \hhH_K\) and define
\(K_0\) as follows:
\[K_0(x,y) = \scalar{P K(x,\cdot)}{P K(y,\cdot)}_{\hhH_K} \qquad (x, y \in X).\]
Then \(K_0\) is a pd kernel such that \(\hhH_{K_0} = E\) and the inner product
of \(\hhH_{K_0}\) coincides with the one on \(E\) inherited from \(\hhH_K\)
(consult, e.g., Theorem~5 on page 37 in \cite{sai}). In particular, \(\ell_2(X)
\cap \hhH_{K_0}\) is dense in \(\hhH_{K_0}\) (and hence \(K_0\) has a pdms
root---see (d) in \THM{exists}) and \LEM{classic} implies that \(K_0 \ll K\).
Now assume \(L \gg 0\) has a pdms root and satisfies \(L \ll K\). Again:
\begin{itemize}
\item condition (d) of \THM{exists} yields that \(\ell_2(X) \cap \hhH_L\) is
 dense in \(\hhH_L\);
\item \LEM{classic} implies that \(\hhH_L \subset \hhH_K\) and the identity
 operator \(I\dd \hhH_L \to \hhH_K\) has norm not exceeding \(1\).
\end{itemize}
It follows from the former of the above properties that \(I(\ell_2(X) \cap
\hhH_L)\) is dense in \(I(\hhH_L)\). Consequently, \(\hhH_L = I(\hhH_L)\) is
contained in the closure (in \(\hhH_K\)) of \(\ell_2(X) \cap \hhH_L\). So,
\(\hhH_L \subset E = \hhH_{K_0}\). Since \(\|I\| \leqsl 1\), we obtain \(L \ll
K_0\) (one more time by \LEM{classic}). The maximality of \(K_0\) (just proved)
implies the uniqueness of \(K_0\).\par
Finally, we turn to (X). We will use here \THM{uniq} (that has not been proved
yet!). A careful reader will verify that the proof of that theorem presented
below is independent of this part of the present result. Equip the vector space
\(H \df \ell_2(X) \cap \hhH_K\) with the inner product
\[\scalar{u}{v}_H \df \scalar{u}{v}_{\ell_2(X)} + \scalar{u}{v}_{\hhH_K} \qquad
(u, v \in H).\]
It is a kind of folklore that the above \(H\) is a Hilbert space on which all
the evaluation functionals (that is, all functions of the form \(u \mapsto
u(x)\) where \(x \in X\)) are continuous (actually both these properties are
easy to prove). This means that \(H\) has a reproducing kernel, say \(K'\).
Since then \(\hhH_{K'} = H \subset \hhH_K\), it follows from \LEM{classic} that
\(K' \ll cK\) for some constant \(c > 0\). We define \(K_1\) as \(\frac1c K'\).
Observe that \(K_1 \ll K\) and \(\hhH_{K_1} = H \subset \ell_2(X) =
\hhH_{\delta_X}\). Another application of \LEM{classic} yields that \(K_1 \ll c'
\delta_X\) for some constant \(c' > 0\). So, \THM{uniq} implies that \(K_1\) has
a unique pdms root. Now let \(L \gg 0\) be such that
\begin{equation}\label{eqn:aux27}
L \ll a K
\end{equation}
for some constant \(a > 0\) and \(L\) has a unique pdms root. Then---again by
\THM{uniq}---\(L \ll a' \delta_X\) for some constant \(a' > 0\). The last
property combined with \LEM{classic} gives \(\hhH_L \subset \hhH_{\delta_X} =
\ell_2(X)\), whereas, similarly, \eqref{eqn:aux27} yields \(\hhH_L \subset
\hhH_K\). So, \(\hhH_L \subset H = \hhH_{K_1}\) and another application of
\LEM{classic} leads to the final conclusion: there exists constant \(b > 0\)
such that \(L \ll b K_1\).
\end{proof}

The next result gives a description of all pdms roots of a pd kernel (that has
such a root).

\begin{pro}{roots}
Let \(K\) be a pd kernel that has a pdms root.
\begin{enumerate}[\upshape(I)]
\item There is a unique positive self-adjoint operator \(B\) in \(\ell_2(X)\)
 that factorizes \(K\) and has \(\ell_{fin}(X)\) as a core.
\item If \(B\) is as specified in \textup{(I)}, then there is a one-to-one
 correspondence \(\kappa\dd \vvV \to \rrR\) between the set \(\vvV\) of all
 linear isometries \(V\dd \overline{\RrR}(B) \to \ell_2(X)\) such that
 the operator \(VB\) is positive, and the set \(\rrR\) of all pdms roots of
 \(K\); \(\kappa\) is given by the rule:
 \begin{equation}\label{eqn:aux24}
 (\kappa(V))(x,y) = \scalar{VBe_x}{e_y}_{\ell_2(X)} \qquad (x, y \in X).
 \end{equation}
\end{enumerate}
\end{pro}
\begin{proof}
The existence of the operator \(B\) with all properties (apart from
the uniqueness) specified in (I) follows from condition (c2) in \THM{exists}. We
fix it and after proving (II) we will show its uniqueness.\par
We turn to (II). Fix for a moment \(V \in \vvV\). It is easy to check that
\(\kappa(V)\) given by \eqref{eqn:aux24} is a pd kernel (because \(VB\) is
positive). Moreover, we have \(\kappa(V)(x,\cdot) = VBe_x\) and hence
\(\kappa(V)\) is an \(\ell_2\)-kernel such that \((\kappa(V) * \kappa(V))(x,y) =
\scalar{VBe_x}{VBe_y}_{\ell_2(X)} = \scalar{Be_x}{Be_y}_{\ell_2(X)} = K(x,y)\).
So, \(\kappa(V) \in \rrR\).\par
Since members of \(\vvV\) are defined (only) on \(\overline{\RrR}(B)\) and
\(\ell_{fin}(X)\) is a core of \(B\), we readily conclude that \(\kappa\) is
one-to-one. So, to end the proof of (II), it remains to check the surjectivity
of \(\kappa\). To this end, let \(L\) be a pdms root of \(K\). We consider
\(L_{\op}\) with target space \(\ell_2(X)\). Observe that for any \(x, y \in
X\), \(\scalar{L_{\op}e_x}{L_{\op}e_y}_{\ell_2(X)} = K(x,y) =
\scalar{Be_x}{Be_y}_{\ell_2(X)}\). This equation implies that there is a linear
isometry \(V\dd \overline{B(\ell_{fin}(X))} \to \ell_2(X)\) such that \(L_{\op}f
= V(Bf)\) for any \(f \in \ell_{fin}(X)\). Since \(\ell_{fin}(X)\) is a core for
\(B\), we get that \(\DdD(V) = \overline{\RrR}(B)\) and \(VB\) is positive (as
it is positive on \(\ell_{fin}(X)\). Consequently, \(V \in \vvV\) and
\[(\kappa(V))(x,y) = \scalar{VBe_x}{e_y}_{\ell_2(X)} =
\scalar{L_{\op}e_x}{e_y}_{\ell_2(X)} = L(x,y) \qquad (x, y \in X).\]\par
Having (II), we can briefly validate the uniqueness of \(B\). Assume \(A\) is
a positive self-adjoint operator in \(\ell_2(X)\) that factorizes \(K\) and has
\(\ell_2(X)\) as a core. Then \(\scalar{Ae_x}{Ae_y}_{\ell_2(X)} = K(x,y) =
\scalar{Be_x}{Be_y}_{\ell_2(X)}\) for any \(x, y \in X\) and it follows from
the previous paragraph that \(A\restriction{\ell_{fin}(X)} =
VB\restriction{\ell_{fin}(X)}\) for some \(V \in \vvV\). Since \(V\) is
isometric and \(\ell_{fin}(X)\) is a core for both \(A\) and \(B\), we obtain
\(A = VB\). Hence \(\NnN(A) = \NnN(B)\). Extend \(V\) to the partial isometry
\(Q\) such that \(\NnN(Q) = \NnN(B)\). Then also \(A = QB\) and it follows from
the uniqueness of the polar decomposition that \(A = B\).
\end{proof}

As an immediate consequence we obtain the following result (we skip its proof).

\begin{cor}{sa}
For any pd kernel that has a pdms root there exists a unique pdms root \(K\)
such that the closure of \(K_{\op}\dd \ell_{fin}(X) \to \ell_2(X)\) is
self-adjoint.
\end{cor}

This is a good moment to give

\begin{proof}[Proof of \THM{uniq}]
Everywhere in this proof \(K_{\op}\) is considered as an operator with target
space \(\hhH_K\). We start from showing that condition (iii) of the theorem is
equivalent to:
\begin{itemize}
\item[(bd)] \(K_{\op}\dd \ell_{fin}(X) \to \hhH_K\) is bounded.
\end{itemize}
Indeed, (iii) says that for some \(c > 0\) we have
\begin{equation}\label{eqn:aux25}
\sum_{x \in X} u(x)\overline{u(y)} K(x,y) \leqsl \delta \sum_{x \in X} |u(x)|^2
\end{equation}
for any \(u \in \ell_{fin}(X)\), which is equivalent to (bd), as the left-hand
side of \eqref{eqn:aux25} coincides with \(\|K_{\op}u\|^2\). Now we turn to
the main part of the proof.\par
Assume (iii) holds. Then the absolute value \(B\) of the closure of \(K_{\op}\)
is positive and bounded (by (bd)). Observe that
\[\scalar{Be_x}{Be_y}_{\ell_2(X)} = \scalar{K_{\op}e_x}{e_y}_{\ell_2(X)} =
K(x,y)\] and thus (e.g. by condition (c1) in \THM{exists}) \(K\) has a pdms
root. Moreover, the above \(B\) witnesses property (I) in \PRO{roots}. So, it
follows from item (II) of that proposition that all other possible pdms roots
are in one-to-one correspondence with linear isometries \(V\dd
\overline{\RrR}(B) \to \ell_2(X)\) such that \(VB\) is positive. But if \(V\) is
such an isometry, then \(VB\) is self-adjoint (being positive and bounded) and
it follows from the uniqueness of \(B\) that \(VB = B\) or, equivalently, that
\(V\) is the identity. This shows (i)---that is, that \(K\) has a unique pdms
root.\par
To prove the reverse implication, we assume that (iii) is false and we will show
that (i) is false as well. To this end, assume \(K\) has a pdms root (otherwise
(i) does not hold). Let \(B\) witness property (I) in \PRO{roots}. We claim that
\(B\) is not bounded. Indeed, since (iii) does not hold, (bd) is false. And we
infer from the proof of \LEM{forall} that \(K_{\op} =
VB\restriction{\ell_{lin}(X)}\) for some linear isometry \(V\dd B(\ell_{fin}(X))
\to \hhH_K\). So, \(B\) is not bounded as \(K_{\op}\) is not such. Now it
follows from \THM{unbd} that there exists a positive (densely defined) operator
\(T\) in \(\ell_2(X)\) such that \(|T| = B \neq T\). Let \(T = QB\) be the polar
decomposition of \(T\). Then \(Q\restriction{\overline{\RrR}(B)}\) is isometric
and differs from the identity map. So, there are at least two pdms roots of
\(K\) thanks to item (II) of \PRO{roots}.\par
Further, observe that it follows from the equivalence of (i) and (iii) (which we
have already proved) that (ii) is implied by (i). So, it remains to show that if
(iii) does not hold, then (ii) is false. To this end, assume there is no \(c >
0\) for which \(K \ll c \delta_X\). Equivalently, \(\hhH_K \not\subset
\hhH_{\delta_X} = \ell_2(X)\) (cf.\ \LEM{classic}). So, there exists a unit
vector \(u \in \hhH_K\) such that \(u \notin \ell_2(X)\). Then \(L \df \bar{u}
\otimes u\) is a pd kernel such that \(L \ll K\) (since \(u\) is a unit vector
in both \(\hhH_K\) and \(\hhH_L\)---see, e.g., Corollary~2 on page 45 in
\cite{sai}) and \(\hhH_L = \lin\{u\}\) where
\begin{equation}\label{eqn:otimes}
\bar{u} \otimes u\dd X \times X \ni (x,y) \mapsto \overline{u(x)} u(y) \in \CCC.
\end{equation}
In particular,
\(\ell_2(X) \cap \hhH_L = \{0\}\) and hence \(L\) does not have a pdms root
(thanks to condition (d) of \THM{exists}).
\end{proof}

\begin{rem}{uncountable}
An inspection of the proofs of \THM[s]{unbd} and \ref{thm:uniq}, combined with
\PRO{iso-lin}, shows that if a pd kernel has at least two pdms roots, then it
actually has uncountably many such roots. We leave the details to interested
readers.
\end{rem}

\begin{cor}{incl}
A pd kernel \(K\) on \(X\) has a unique pdms root iff \(\hhH_K \subset
\ell_2(X)\).
\end{cor}
\begin{proof}
The assertion immediately follows from \THM{uniq} and \LEM{classic}.
\end{proof}

\begin{cor}{sum}
Let \(K\) be a pd kernel on \(X\) that has a unique pdms root. For any pd kernel
\(L\) on \(X\), \(L\) has a pdms root iff so has \(K+L\).
\end{cor}
\begin{proof}
The proof of \THM{uniq} shows that \(T \df K_{\op}\dd \ell_{fin}(X) \to H \df
\hhH_K\) is bounded. Let \(R\dd \ell_{fin}(X) \to E\) be any operator that
factorizes \(L\). Then \(T \oplus R\dd \ell_{fin}(X) \to H \oplus E\) factorizes
\(K+L\) (as \(T\) factorizes \(K\)). Since \(T\) is bounded, \(R\) is closable
iff so is \(T \oplus R\). Thus, the conclusion follows from condition (c1) of
\THM{exists} and \LEM{forall}.
\end{proof}

\begin{exm}{otimes}
Simplest possible pd kernels on a set \(X\) are of the form \eqref{eqn:otimes}
where \(u\dd X \to \CCC\) is totally arbitrary. (The simplicity of these kernels
can be justified as follows: they are precisely those pd kernels \(K\) for which
\(\dim(\hhH_K) \leqsl 1\).) Since \(\hhH_{\bar{u} \otimes u} = \lin \{u\}\),
we either have \(\hhH_{\bar{u} \otimes u} \subset \ell_2(X)\) or \(\ell_2(X)
\cap \hhH_{\bar{u} \otimes u} = \{0\}\). So, condition (d) of \THM{exists} and
\COR{incl} imply that for a function \(v\dd X \to \CCC\) \tfcae
\begin{enumerate}[(i)]
\item \(\bar{v} \otimes v\) has a pdms root;
\item \(\bar{v} \otimes v\) has a unique pdms root;
\item \(v \in \ell_2(X)\).
\end{enumerate}
In \EXM{uniform} we will use the above charecterization to give
a (counter)example witnessing that the uniform limit of bounded pd kernels
having a unique pdms root can have no pdms roots.
\end{exm}

To simplify further statements, we introduce a few additional notions:

\begin{dfn}{diag}
A pd kernel \(K\) on \(X\) is said to be
\begin{itemize}
\item \textit{diagonal} if \(K(x,y) = 0\) for any distinct \(x, y \in X\);
\item \textit{pointwise countable} if for any \(x \in X\) the set \(\{y \in X\dd
 K(x,y) \neq 0\}\) is (at most) countable.
\end{itemize}
A \textit{rescaling} of \(K\) is a pd kernel of the form
\begin{equation}\label{eqn:scale}
X \times X \ni (x,y) \mapsto \overline{u(x)} u(y) K(x,y) \in \CCC
\end{equation}
where \(u\dd X \to \CCC\) is abitrary. The kernel given by \eqref{eqn:scale} is
called \textit{\(u\)-rescaling}. The \(u\)-rescaling is \textit{non-vanishing}
if \(u\) has no zeros.
\end{dfn}

Our nearest aim is to characterize those kernels that admit a non-vanishing
rescaling having a pdms root. The following result will serve as a useful tool
in investigating this issue. Everywhere below \(\NNN\) denotes the set of all
positive integers.

\begin{lem}{scale}
Let \(K\) be a pd kernel on \(\NNN\) and \(D\) be the diagonal pd kernel such
that \(D(n,n) = n^2\) for any \(n\).
\begin{enumerate}[\upshape(I)]
\item If \(\sum_{n,m=1}^{\infty} |K(n,m)|^2 \leqsl 1\), then \(K \ll
 \delta_{\NNN}\).
\item If \(\sum_{n=1}^{\infty} K(n,n) \leqsl 1\), then \(K \ll \delta_{\NNN}\).
\item If \(|K|\) is upper bounded by \(c > 0\), then \(K \ll \frac{c \pi^2}{6}
 D\).
\end{enumerate}
\end{lem}
\begin{proof}
Although properties (I) and (II) are consequences of well-known results on
Schatten class and Hilbert-Schmidt operators (consult, e.g., \cite{sch} or
the material of \S18 in \cite{co2}), below we present their brief proofs.\par
To show (I), for any \(\alpha \in \ell_{fin}(\NNN)\) the Schwarz inequality
yields
\begin{multline*}
\sum_{n,m=1}^{\infty} \alpha(n)\overline{\alpha(m)} K(n,m) \leqsl
\sqrt{\sum_{n,m=1}^{\infty} |\alpha(n)\overline{\alpha(m)}|^2} \cdot
\sqrt{\sum_{n,m=1}^{\infty} |K(n,m)|^2}\\\leqsl
\sqrt{\sum_{n=1}^{\infty} |\alpha(n)|^2 \cdot \sum_{m=1}^{\infty}
|\alpha(m)|^2} = \sum_{n=1}^{\infty} |\alpha(n)|^2
\end{multline*}
which is equivalent to \(K \ll \delta_{\NNN}\). Further, (II) is a consequence
of (I) since
\[\sum_{n,m=1}^{\infty} |K(n,m)|^2 \leqsl \sum_{n,m=1}^{\infty} K(n,n) K(m,m) =
\Bigl(\sum_{n=1}^{\infty} K(n,n)\Bigr)^2\]
(cf.\ the paragraph following \THM{pdms}). We turn to (III). For simplicity,
denote the \(u\)-rescaling of a pd kernel \(L\) by \(L_u\). In particular,
\((L_u)_v = L_{uv}\). Define \(u\dd \NNN \ni n \mapsto
\frac{\sqrt{6}}{\pi n\sqrt{c}} \in (0,\infty)\). Since \(|K|\) is bounded by
\(c\) and \(\sum_{n=1}^{\infty} n^{-2} = \frac{\pi^2}{6}\), we easily get that
\(\sum_{n=1}^{\infty} K_u(n,n) \leqsl 1\). So, we conclude from (II) that \(K_u
\ll \delta_{\NNN}\). Consequently, \(K = (K_u)_{\frac1u} \ll
(\delta_{\NNN})_{\frac1u} = \frac{c \pi^2}{6} D\) and we are done.
\end{proof}

Now we can characterize pd kernels admitting non-vanishing rescalings having
pdms roots.

\begin{thm}{scale}
For a pd kernel \(K\) on \(X\) \tfcae
\begin{enumerate}[\upshape(i)]
\item \(K\) has a non-vanishing rescaling that has a pdms root;
\item the \(v\)-rescaling of \(K\) has a unique pdms root for some \(v\dd X \to
 (0,1)\);
\item \(K \ll D\) for some diagonal pd kernel \(D\) on \(X\);
\item \(K \ll C\) for some pointwise countable pd kernel \(C\) on \(X\);
\item \(K\) is pointwise countable.
\end{enumerate}
In particular,
\begin{itemize}
\item each pd kernel on a countable set satisfies conditions \textup{(ii)} and
 \textup{(iii)};
\item the collection \(\ccC\) of all pd kernels \(K\) on \(X\) that satisfy
 \textup{(i)} is a convex cone such that \(L \in \ccC\) whenever \(0 \ll L \ll
 K\) for some \(K \in \ccC\) or \(L\) is the pointwise product of two members of
 \(\ccC\).
\end{itemize}
\end{thm}
\begin{proof}
Since the additional claim of the theorem easily follows from the equivalence of
(i), (iii) and (v) (recall that---according to the Schur's theorem
\cite{shu}---the pointwise product, usually called the Hadamard product, of two
pd kernels defined on a common set is a pd kernel as well; see also Section~3 of
Chapter~I in \cite{don}---consult the material on page~9 therein), we only need
to show the equivalence of all conditions (i)--(v). As done in the previous
proof, we shall denote, for simplicity, the \(u\)-rescaling of a pd kernel \(L\)
by \(L_u\). First we will show the following implications:
(iii)\(\implies\)(ii)\(\implies\)(i)\(\implies\)(v)\(\implies\)(iii) and then
we will briefly deduce the equivalence of (iv) and (iii).\par
Assume (iii) holds and let \(D\) be as specified therein. Let \(v\dd X \to
(0,1)\) be given by \(v(x) \df 1/\max(\sqrt{D(x,x)},2)\ (x \in X)\). Since \(K
\ll D\), we get \(K_v \ll D_v\) and easily \(D_v \ll \delta_X\). Hence \(K_v \ll
\delta_X\) and it follows from \THM{uniq} that \(K_v\) (or \(v\)) witnesses
(ii). Now observe that (i) implies (v) by condition (b) of \THM{exists}, and
obviously follows from (ii). So, we now assume that (v) holds and will show that
(iii) is fulfilled. First of all, observe that \(K\), being pointwise countable,
induces a decomposition of \(X = \bigsqcup_{s \in S} X_s\) into at most
countable (pairwise disjoint) sets \(X_s\) such that \(K = \bigoplus_{s \in S}
K^{(s)}\) where \(K^{(s)} = K\restriction{X_s \times X_s}\ (s \in S)\) (cf.\
the proof that (b) follows from (a1) in \THM{exists}). Since \(X_s\) is
countable, there exists a function \(u_s\dd X_s \to (0,1)\) such that
\(\sum_{x \in X_s} u_s(x)^2 K^{(s)}(x,x) \leqsl 1\). We infer from \LEM{scale}
that \((K^{(s)})_{u_s} \ll \delta_{X_s}\). Consequently, \(\bigoplus_{s \in S}
(K^{(s)})_{u_s} \ll \delta_X\). Now it suffices to define \(u\dd X \to (0,1)\)
as the union of all \(u_s\) (that is, \(u\restriction{X_s} = u_s\)) and \(D\) as
\((\delta_X)_{\frac1u}\) to get \(K = (K_u)_{\frac1u} =
\bigl(\bigoplus_{s \in S} (K^{(s)})_{u_s}\bigr)_{\frac1u} \ll
(\delta_X)_{\frac1u} = D\).\par
It remains to explain why (iv) is equivalent to (iii). Since diagonal pd kernels
are pointwise countable, (iv) follows from (iii). Conversely, if (iv) holds and
\(C\) is as specified therein, then we know that there exists a diagonal pd
kernel \(D\) such that \(C \ll D\) (as \(C\) satisfies (v)). Then clearly \(K
\ll D\) and we are done.
\end{proof}

Property (VI) of \THM{pdms} implies that if \(K\) and \(L\) are pd kernels on
\(X\) that have pdms roots, then the kernel \(K+L\) has a pdms root as well.
Conversely, a basic consequence of \THM{uniq} is that if \(K+L\) has a unique
pdms root, so have both \(K\) and \(L\). It turs out that a counterpart of this
property for pd kernels having at least two pdms roots is (in a sense
always---see \PRO{two} below) false and when the kernel acts on a countable set,
the property mentioned above crashes in a striking way (see item (A) in
\PRO{two}). To be more precise, we call a pd kernel \textit{free of pdms roots}
if no non-zero pd kernel \(L \ll K\) has a pdms root. With the aid of
the results of Section~3, we now show that

\begin{pro}{two}
Let \(K\) be a pd kernel on \(X\) that has at least two pdms roots.
\begin{enumerate}[\upshape(A)]
\item If \(X\) is countable, \(K\) is a sum of two pd kernels each of which is
 free of pdms roots.
\item \(K\) is a sum of two pd kernels none of which has a pdms root.
\end{enumerate}
\end{pro}

In the proof we shall make use of the following

\begin{lem}{free}
A pd kernel \(K\) on \(X\) is free of pdms roots iff \(\ell_2(X) \cap \hhH_K =
\{0\}\).
\end{lem}
\begin{proof}
Let \(K_0\) be as specified in property (IX) listed in \THM{pdms}. Observe that
\(K\) is free of pdms roots iff \(K_0 \equiv 0\). An inspection of the proof of
the aforementioned item (IX) shows that \(\hhH_{K_0}\) coincides with
the closure in \(\hhH_K\) of \(\ell_2(X) \cap \hhH_K\), from which
the conclusion follows.
\end{proof}

\begin{proof}[Proof of \PRO{two}]
We start from (A). It follows from the countability of \(X\) that \(\hhH_K\) is
separable. Since \(K\) has a pdms root, \(\RrR \df \ell_2(X) \cap \hhH_K\) is
dense in \(\hhH_K\) (by \THM{exists}). Moreover, \(\RrR \neq \hhH_K\), because
\(K\) has more than one pdms root (see \COR{incl}). Note that \(\RrR\) is
an operator range in \(\hhH_K\) (the proof of (X) in \THM{pdms} shows that
\(\RrR\) admits a Hilbert space norm stronger than the norm induced from
\(\hhH_K\)). So, condition (g) of \THM{pure} implies that there is a closed
linear subspace \(\wwW\) of \(\hhH_K\) such that \(\wwW \cap \RrR = \wwW^{\perp}
\cap \RrR = \{0\}\). Equivalently,
\begin{equation}\label{eqn:aux28}
\wwW \cap \ell_2(X) = \wwW^{\perp} \cap \ell_2(X) = \{0\}.
\end{equation}
Denoting by \(P\) and \(Q\) the orthogonal projections in \(\hhH_K\) onto,
respectively, \(\wwW\) and \(\wwW^{\perp}\), we define pd kernels \(K_1\) and
\(K_2\) on \(X\) by \(K_1(x,y) \df \scalar{PK(x,\cdot)}{PK(y,\cdot)}_{\hhH_K}\)
and \(K_2(x,y) \df \scalar{QK(x,\cdot)}{QK(y,\cdot)}_{\hhH_K}\). Then
\(\hhH_{K_1} = \wwW\) and \(\hhH_{K_2} = \wwW^{\perp}\) (see Theorem~5 on page
37 in \cite{sai}) and therefore---thanks to \eqref{eqn:aux28} and
\LEM{free}---both \(K_1\) and \(K_2\) are free of pdms roots. But \(K_1+K_2 =
K\) (since \(P+Q = I\)) and we are done.\par
Now we pass to the general case (item (B)). We may and do assume \(X\) is
uncountable. We will show that \(X\) contains a countable subset \(A\) such that
\[K = (K\restriction{A \times A}) \oplus
\bigl(K\restriction{(X \setminus A) \times (X \setminus A)}\bigr)\] and
\(K\restriction{A \times A}\) has at least two pdms roots. Assume for a moment
that we have already found such a set \(A\). It then follows from (A) that
\[K\restriction{A \times A} = L_1+L_2\] where \(L_1\) and \(L_2\) are pd kernels
on \(A\) without pdms roots. We define pd kernels \(K_1\) and \(K_2\) on \(X\)
by: \[K_1 = L_1 \oplus
\bigl(K\restriction{(X \setminus A) \times (X \setminus A)}\bigr)\] and \(K_2 =
L_2 \oplus 0\) (where \(0\) here means the zero function on \((X \setminus A)
\times (X \setminus A)\)). It is easily seen that both \(K_1\) and \(K_2\) are
pd kernels such that \(K = K_1+K_2\). Moreover, it follows from property (I) in
\THM{pdms} that none of \(K_1\) and \(K_2\) has a pdms root. (Indeed, if \(K_j\)
had a pdms root, so would have \(K_j\restriction{A \times A} = L_j\) which is
impossible.) Hence, it remains to construct the set \(A\) with appropriate
properties. We will do this with the aid of condition (b) of \THM{exists}.\par
It follows from the aforementioned result that \(K = \bigoplus_{s \in S} K_s\)
where for any \(s \in S\), \(K_s\) is a pd kernel on a countable set \(X_s\)
that has a pdms root. If one of \(K_s\) has at least two pdms roots, we just set
\(A = X_s\) and finish the proof. So, we assume that \(K_s\) has a unique pdms
root for any \(s \in S\). It follows from \THM{uniq} that
\begin{equation}\label{eqn:aux29}
K_s \ll c_s \delta_{X_s}
\end{equation}
for some constant \(c_s \geqsl 0\). We take \(c_s\) to be the smallest possible
non-negative number witnessing \eqref{eqn:aux29}. We claim that
\begin{equation}\label{eqn:aux30}
\sup_{s \in S} c_s = \infty.
\end{equation}
Indeed, if \(c \df \sup_{s \in S} c_s\) was finite, then we would have \(K_s \ll
c \delta_{X_s}\) for all \(s \in S\) and thus also \(K \ll c \delta_X\), which
would mean that \(K\) would have a unique pdms root (by \THM{uniq}). So,
\eqref{eqn:aux30} holds and therefore there exists a sequence \(s_1,s_2,\ldots
\in S\) for which \(\lim_{n\to\infty} c_{s_n} = \infty\). We set \(A \df
\bigcup_{n=1}^{\infty} X_{s_n}\). It follows from property (I) in \THM{pdms}
that \(K\restriction{A \times A}\) has a pdms root. Further, since the numbers
\(c_s\) have been chosen optimal, there is no \(c > 0\) such that
\(K\restriction{A \times A} \ll c \delta_A\). Consequently,
\(K\restriction{A \times A}\) has more than one pdms root (by \THM{uniq}) and
we are done.
\end{proof}

Since any pd kernel defined on a finite set has a (unique) pdms root, condition
(b) of \THM{exists} completely reduces the study of the class of pd kernels
having pdms roots to pd kernels on infinite countable sets. So, pd kernels on
\(\NNN\) (or \(\ZZZ\)) are of special interest. Below we give a single positive
example on the existence of pdms roots and a list of counterexamples to related
subjects. Everywhere below \(\ell_2 = \ell_2(\NNN)\) (that is, indices of
sequences in \(\ell_2\) start from \(1\)), \(\ell_{fin} = \ell_{fin}(\NNN)\) and
\(e_1,e_2,\ldots\) is the canonical basis of \(\ell_2\).

\begin{exm}{shift}
Let \(S \in \Bb(\ell_2)\) be the standard unilateral shift (that is, \(S\) is
a linear isometry such that \(S e_n = e_{n+1}\)) and \(f \in \ell_2\) be
arbitrary. Then the pd kernel \(K\) on \(\NNN\) given by:
\begin{equation}\label{eqn:shift}
K(n,m) = \scalar{S^{n-1}f}{S^{m-1}f}_{\ell_2} \qquad (n, m \in \NNN)
\end{equation}
has a pdms root. To prove this, it is sufficient to find a closable operator
that factorizes \(K\) (see item (c1) in \THM{exists}). To this end, we model
the shift \(S\) on the Hardy space \(H^2\) of holomorphic functions on the disc
\(\DDD = \{z \in \CCC\dd\ |z| < 1\}\). Recall that:
\begin{enumerate}[(\(H^2\)a)]
\item a holomorphic function \(u\dd \DDD \ni z \mapsto \sum_{k=0}^{\infty} a_k
 z^k \in \CCC\) belongs to \(H^2\) if \(\sum_{k=0}^{\infty} |a_k|^2 < \infty\);
\item the monomials \(1,z,z^2,\ldots\) form an orthonormal basis of \(H^2\) (we
 use here a standard simplified notation; \(z^k\) is in fact the function
 \(z \mapsto z^k\) restricted to \(\DDD\));
\item for any \(z \in \DDD\), the evaluation functional \(H^2 \ni u \mapsto u(z)
 \in \CCC\) is continuous (the reproducing kernel of \(H^2\) has the form
 \((z,w) \mapsto \frac{1}{1-\bar{z}w}\));
\item the operator \(M\dd H^2 \to H^2\) given by \((Mu)(z) = zu(z)\ (z \in
 \DDD,\ u \in H^2)\) is unitarily equivalent to \(S\); more precisely, if \(V\dd
 \ell_2 \to H^2\) is a unitary operator such that \(Ve_n = z^{n-1}\ (n > 0)\),
 then \(V^{-1} M V = S\).
\end{enumerate}
Now for \(f \in \ell_2\) set \(F \df Vf \in H^2\) and define \(T\dd
\lin\{z^k\dd\ k \geqsl 0\} \to H^2\) as the multiplication operator by \(F\)
(that is, \((Tp)(z) = F(z)p(z)\) for \(p \in \DdD(T)\) and \(z \in \DDD\)). It
easily follows from the above property (\(H^2\)c) that \(T\) is closable.
Therefore \(V^{-1}TV\dd \ell_{fin} \to \ell_2\) is closable as well. Note that
\(S^k = V^{-1} M^k V\) and hence \(S^k f = V^{-1} M^k F = V^{-1} T z^k =
V^{-1}TV e_{k+1}\). Consequently, \(V^{-1}TV\) factorizes \(K\) and we are done.
\end{exm}

\begin{exm}{c0}
Similarly to the notion of an \(\ell_2\)-kernel, let us call a kernel \(K\) on
\(\NNN\) a \textit{\(c_0\)-kernel} if \(\lim_{n\to\infty} (|K(p,n)|+|K(n,p)|) =
0\) for all \(p \in \NNN\). Item (IV) in \THM{pdms} says that each
\textbf{bounded} pd kernel on \(\NNN\) that has a pdms root is
a \(c_0\)-kernel. On the other hand, pd \(\ell_2\)-kernels always have pdms
roots. So, two natural (contrary) questions arise:
\begin{enumerate}[(Quest{i}on A)]
\item Does every bounded pd \(c_0\)-kernel on \(\NNN\) have a pdms root?
\item Is every bounded pd kernel that has a pdms root an \(\ell_2\)-kernel?
\end{enumerate}
In this example we answer these two questions in the negative. Both the kernels
constructed below will be constant on the diagonal \(\{(n,n)\dd\ n \in \NNN\}\).
\begin{enumerate}[(A)]
\item In this part we construct a pd kernel \(K\) on \(\NNN\) that has a pdms
 root and satisfies:
 \begin{itemize}
 \item[\(\bullet\)] \(K(n,n) = 1\) for all \(n \in \NNN\) (and thus \(K\) is
  bounded);
 \item[\(\bullet\)] \(K(1,\cdot) \notin \bigcup_{p>0} \ell_p\).
 \end{itemize}
 Define a sequence \(f = (a_1,a_2,\ldots)\) of positive real numbers
 as follows:
 \[a_n = \sqrt{\frac{1}{H_n}-\frac{1}{H_{n+1}}} =
 \frac{1}{\sqrt{(n+1)H_n H_{n+1}}}\]
 where \(H_n \df \sum_{k=1}^n \frac1k\). The second formula for \(a_n\) shows
 that the sequence \(f\) is monotone decreasing, whereas the first implies that
 the series \(\sum_{k=1}^{\infty} a_k^2\) is telescoping and its sum equals
 \(\frac{1}{H_1} = 1\). So, \(f \in \ell_2\). Let \(S \in \Bb(\ell_2)\) be
 the shift as specified in \EXM{shift}. That example shows that the pd kernel
 \(K\) given by \eqref{eqn:shift} has a pdms root. We now check that \(K\) has
 all announced properties. It is easily seen that \(K(n,n) = \|f\|^2 = 1\).
 Finally, for any \(n > 0\) we have (recall that \(f\) is monotone decreasing):
 \[K(1,n) = \sum_{k=1}^{\infty} a_k a_{k+n-1} \geqsl \sum_{k=1}^{\infty}
 a_{k+n-1}^2 = \frac{1}{H_n}.\]
 Since \(\lim_{n\to\infty} \frac{\log n}{H_n} = 1\), we get that
 \(\sum_{n=1}^{\infty} \frac{1}{H_n^p} = \infty\) for any \(p > 0\) and
 therefore \(K(1,\cdot) \notin \bigcup_{p>0} \ell_p\).
\item In contrast to the example given in (A), now we construct a pd kernel
 \(K\dd \NNN \times \NNN \to [0,1]\) such that \(K(m,\cdot) \in \bigcap_{p>2}
 \ell_p\) for any \(m \in \NNN\), but \(K\) has no pdms roots. To this end, we
 take any sequence \(u \in (\bigcap_{p>2} \ell_p) \setminus \ell_2\) whose all
 entries lie in \([0,1]\) and define \(K_0\) as \(u \otimes u\) (see
 \eqref{eqn:otimes}). It follows from \EXM{otimes} that \(K_0\) has no pdms
 roots. However, \(K_0(m,\cdot) = u(m) u \in \bigcap_{p>2} \ell_p\). Now let
 \(D\) be the diagonal pd kernel on \(\NNN\) such that \(D(n,n) = 1-u(n)^2\).
 Since \(D \ll \delta_{\NNN}\), we conclude that \(D\) has a unique pdms root.
 So, it follows from \COR{sum} that \(K \df K_0 + L\) has no pdms roots. This
 kernel satisfies \(K(n,n) = 1\) for all \(n \in \NNN\) and \(K(m,\cdot) \in
 \bigcap_{p>2} \ell_p\) for any \(m > 0\).
\end{enumerate}
The above examples suggest that there is no handy description of bounded pd
kernels that have a pdms root.
\end{exm}

\begin{exm}{uniform}
Properties (VI) and (VII) listed in \THM{pdms} suggest that perhaps pd kernels
on a given set that have pdms roots form a set closed in the pointwise or
uniform topology (in the space of all kernels). As the following simple example
shows, this is not the case.\par
Let \(u\dd \NNN \to [0,1]\) be given by \(u(n) = \frac{1}{\sqrt{n}}\) and \(K
\df u \otimes u\) (see \eqref{eqn:otimes}). We infer from \EXM{otimes} that
\(K\) has no pdms roots. Moreover, since \(\dim(\hhH_K) = 1\), \(K\) is free of
pdms roots. Now for \(n > 0\) let \(K_n\) be a pd kernel on \(\NNN\) such that
\(K_n(p,q) = K(p,q)\) if both \(p\) and \(q\) are less than \(n\), and
\(K_n(p,q) = 0\) otherwise. It is easy to see that \(K_n\) is a pd kernel.
Since \(K_n\) is supported on a finite set, it has a unique pdms root (e.g., by
\THM{uniq}). However, since \(\lim_{n\to\infty} u(n) = 0\), \(K_n\) uniformly
converge to \(K\). So, the uniform limit of a bounded sequence of pd kernels
each of which has a unique pdms root can be free of pdms roots.
\end{exm}

We end the paper with the following

\begin{exm}{bd}
Property (III) listed in \THM{pdms} gives a necessary condition for an unbounded
pd kernel to have a pdms root that reads as follows: if an unbounded pd kernel
has a pdms root, so have all its bouded rescalings. In this example we show that
it is insufficient. More precisely, we will construct an unbounded pd kernel
\(K\) on \(\NNN\) such that:
\begin{itemize}
\item each bounded rescaling of \(K\) has a unique pdms root;
\item \(K\) has no pdms roots.
\end{itemize}
To construct such a kernel, it is sufficient to find two self-adjoint bounded
operators \(A, B \in \Bb(H)\) on a separable Hilbert space \(H\) such that for
some orthonormal basis \(f_1,f_2,\ldots\) of \(H\) the following conditions
hold:
\begin{enumerate}[({a}ux1)]
\item \(\RrR(A) \cap \RrR(B) = \{0\}\);
\item \(A f_n = r_n f_n\) where \(0 < r_n \to 0\) as \(n \to \infty\);
\item \(\|B f_n\| = 1\) for any \(n > 0\).
\end{enumerate}
To convince oneself of that, assume we have found such operators \(A\) and
\(B\). Observe that (aux2) implies that \(\NnN(A) = \{0\}\) (and thus
\(\RrR(A)\) is dense in \(H\)). We set \(\DdD \df \RrR(A)\) and \(C \df
A^{-1}\dd \DdD \to H\). Additionally, let \(U\dd \ell_2 \to H\) be the unitary
operator such that \(U e_n = f_n\). The kernel \(K\) we search for is given by:
\[K(n,m) \df \scalar{BCUe_n}{BCUe_m}_H = \frac{\scalar{Bf_n}{Bf_m}_H}{r_n r_m}
\qquad (n, m \in \NNN)\]
(the latter formula follows from (aux2)). To show that \(K\) has no pdms roots,
it suffices to check that the operator \(BCU\restriction{\ell_{fin}}\) is not
closable (as \(BCU\) factorizes \(K\); see condition (c1) in \THM{exists} and
\LEM{forall}). Since \(U\) is unitary, we only need to verify that \(T \df
BC\restriction{\lin\{f_n\dd\ n > 0\}}\) is not closable. To this end, first note
that the graph \(\Gamma(BC)\) of \(BC\) is contained in the closure (in \(H
\times H\)) of \(\Gamma(T)\) (as \(B\) is bounded and \(\DdD(T)\) is a core for
\(C\)). Thus, \(T\) is closable iff so is \(BC\), iff \((BC)^*\) is densely
defined. But \((BC)^* = C^* B^* = CB\) (since \(B\) is bounded and both \(B\)
and \(C\) are self-adjoint) and \(\DdD(CB) = \{0\}\) by (aux1). So, \(T\) is not
closable and hence \(K\) has no pdms roots. Now let \(v\dd \NNN \to \CCC\) be
any function such that the \(v\)-rescaling of \(K\) is bounded. This means that
\(\sup_{n\in\NNN} |v(n)|^2 K(n,n) < \infty\). But \(K(n,n) =
\bigl(\frac{\|Bf_n\|}{r_n}\bigr)^2 = r_n^{-2}\) (by (aux3)). So, there exists
\(c > 0\) such that
\begin{equation}\label{eqn:aux31}
\frac{|v(n)|}{r_n} \leqsl c \qquad (n > 0).
\end{equation}
In particular, \(\lim_{n\to\infty} v(n) = 0\). Denote by \(W\) the diagonal
operator on \(\ell_2\) (with respect to the orthonormal basis) such that \(We_n
= \overline{v(n)} e_n\) (\(W\) is compact, but we do not need this property).
Since \(\DdD = \bigl\{\sum_{n=1}^{\infty} \alpha_n f_n\dd\ \sum_{n=1}^{\infty}
(|\alpha_n|/r_n)^2 < \infty\bigr\}\), we easily infer from \eqref{eqn:aux31}
that \(\RrR(UW) \subset \DdD\) and \(CUW\) is bounded (here we do not apply
the Closed Graph Theorem---the boundedness of \(CUW\) is a direct consequence of
\eqref{eqn:aux31}, (aux2) and the definitions of \(C\), \(U\) and \(W\)).
Consequently, \(Q \df BACW\) is bounded and \[\scalar{Qe_n}{Qe_m}_H =
\overline{v(n)}v(m) \scalar{BCUe_n}{BCUe_m}_H = \overline{v(n)}v(m) K(n,m).\]
This means that the \(v\)-rescaling \(L\) of \(K\) is factorized by a bounded
operator (namely, by \(Q\)), from which it easily follows that \(L\) satisfies
condition (iii) of \THM{uniq} and therefore has a unique pdms root. In this way
we have reduced our proof to the construction of self-adjoint operators \(A\)
and \(B\) that satisfy (aux1)--(aux3), which we do below.\par
First of all, let \(H\) stand for \(L^2([0,1])\) (with the Lebesgue measure). We
arrange all integers in a one-to-one sequence \(k_1,k_2,\ldots\) and define
\(f_1,f_2,\ldots\) as a rearranged standard exponential orthonormal basis of
\(H\): \(f_n(x) = e^{2k_n\pi x\textup{i}}\). Let \(A \in \Bb(H)\) be given by
\(A f_n = \frac1n f_n\). So, \(A g = \sum_{n=1}^{\infty}
\frac{\scalar{g}{f_n}_H}{n} f_n\) for any \(g \in H\). Since the series
\(\sum_{n=1}^{\infty} \frac{\scalar{g}{f_n}_H}{n}\) is absolutely convergent, we
see that \(\RrR(A)\) consists of continuous functions. Note also that (aux2)
holds with \(r_n = \frac1n\).\par
Now we will define \(B\). To this end, take any bounded Borel function \(u\dd
[0,1] \to (0,\infty)\) such that for any non-empty open interval \(I \subset
[0,1]\),
\begin{equation}\label{eqn:int}
\int_I \frac{\dint{t}}{u(t)^2} = \infty.
\end{equation}
The proof that such a function exists is left to the reader. Additionally, we
may and do assume that \(u\) is a unit vector in \(H\). We define \(B\) as
the multiplication operator by \(u\); that is, \(B\dd H \to H\) is given by
\((Bf)(t) = u(t)f(t)\). It is clear  that \(B\) is positive self-adjoint and
bounded, and has trivial kernel. We claim that \(\RrR(A) \cap \RrR(B) = \{0\}\).
Indeed, it is sufficient to show that \(\RrR(B)\) contains no non-zero
continuous functions. To this end, assume that \(g\dd [0,1] \to \CCC\) is
continuous and non-zero. Then we can find a non-empty open interval \(I \subset
[0,1]\) and a constant \(\epsi > 0\) such that \(|g(x)| \geqsl \epsi\) for all
\(x \in I\). But then \(\int_0^1 \bigl(\frac{|g(x)|}{u(x)}\bigr)^2 \dint{x}
\geqsl \int_I \frac{\epsi^2}{u(x)^2} \dint{x} = \infty\) (by \eqref{eqn:int})
and hence \(\frac{g}{u} \notin H\). We conclude that \(g \notin \RrR(B)\).
Finally, since \(\|u\| = 1\) and \(|f_n(x)| = 1\) for any \(x \in I\), condition
(aux3) also holds and the proof is complete.
\end{exm}

\end{document}